\documentclass{amsart}
\usepackage{amssymb}
\usepackage{amsthm}
\usepackage{tikz}
\usepackage{hyperref}
\newtheorem{theorem}{Theorem}

\newtheorem{lemma}{Lemma}

\newtheorem{proposition}{Proposition}
\newtheorem{corollary}{Corollary}

\begin{document}

\title{A Spiegelungssatz for function field $L$-series}

\author{Bruno Angl\`es   \and Floric Tavares Ribeiro}

\address{
Universit\'e de Caen Basse-Normandie, 
Laboratoire de Math\'ematiques Nicolas Oresme, 
CNRS UMR 6139, 
Campus II, Boulevard Mar\'echal Juin, 
B.P. 5186, 
14032 Caen Cedex, France.
}
\email{bruno.angles@unicaen.fr, floric.tavares-ribeiro@unicaen.fr}

\begin{abstract} We prove a kind of reflection principle for certain non-archimedean $L$-series in positive characteristic. We also prove the pseudo-cyclicity and pseudo-nullity of certain several variable generalizations of the class modules introduced by L. Taelman in 2010. \end{abstract}
\date{ february, 2,  2015}
\maketitle
\tableofcontents
\section {Introduction}${}$\par
The existence of a functional equation for David Goss $L$-series (see \cite{GOS}, chapter 8 and  \cite{GOS4}) is an open  problem in the arithmetic theory of function fields over finite fields. A hint that such a functional identity should exist can be  found for example  in the statement of the  Spiegelungssatz for the Carlitz module (see \cite{ANG&TAE2}) that asserts that Taelman's class module (see \cite{TAE1}, \cite{TAE3}) and some isotypic components of the $p$-Sylow subgroup of the $\mathbb F_q$- rational points of some jacobians attached to the Carlitz module ($\mathbb F_q$ being a finite field of characteristic $p$ and having $q$ elements) are related by a kind of reflection principle similar to the classical Spiegelungssatz of Leopoldt (\cite{LEO}). Indeed, since the  work of L. Taelman (see \cite{TAE2}), it is known that the "size" of Taelman's class modules are related to the values at one  of David Goss $L$-series (see also \cite{APTR}, \cite{ANG&TAE1}, \cite{DEM}, \cite{FAN}), and, by the results of D. Goss and W. Sinnott (\cite{GOS&SIN}), the non-triviality of the isotypic components of the $p$-Sylow subgroup of the $\mathbb F_q$- rational points of some jacobians attached to the Carlitz module is connected to the values of the Carlitz-Goss zeta function at negative integers (see also \cite{ABBL}). In this article we show that there exists a kind of  reflection principle  for certain  $v$-adic $L$-series attached to $K:=\mathbb F_q(\theta),$ $\theta$ being an indeterminate over the finite field $\mathbb F_q,$  $v$ being  any place of $K.$\par
Let $s\geq 0$ be an integer and let $t_1, \ldots, t_s,z$ be $s+1$ indeterminates over $K.$ Let's consider the following power series:
$$L(t_1, \ldots, t_s;z)=\sum_{k\geq0}\sum_{a\in A_{+,k}}\frac{a(t_1)\cdots a(t_s)}{a}z^k \in K(t_1, \ldots, t_s)[[z]],$$
where $A_{+,k}$ denotes the monic elements in $A:=\mathbb F_q[\theta]$ of degree $k.$ Note that:
$$L(t_1, \ldots, t_s; z)=\prod_{P{\rm \, prime \, of\, }A}(1-\frac{P(t_1)\cdots P(t_s)z^{deg_{\theta}P}}{P})^{-1},$$
where by a prime of $A$ we mean a monic irreducible polynomial in $A.$ We observe that $L(t_1, \ldots, t_s;z)$ is of adelic nature, indeed $L(t_1, \ldots, t_s;z)$ defines an $\infty$-adic $s+1$-variable entire function (\cite{ANG&PEL1}, Proposition 6), and for every prime $P$ of $A,$ by Theorem \ref{Preltheorem2}, $(1-\frac{P(t_1)\cdots P(t_s)z^{deg_{\theta}P}}{P})L(t_1, \cdots, t_s; z)$ defines a $P$-adic $s+1$-variable entire function. Our aim in this paper is to study this adelic $L$-series at $z=1$ or its derivative (with respect to $z$) at $z=1$ in the case of a trivial zero at $z=1.$\par

We first recall the case of the place $\infty$ and $s=1.$ Set $t=t_1$ and let  $\mathbb C_{\infty}$ be the completion of a fixed algebraic closure of $K_{\infty}:=\mathbb F_q((\frac{1}{\theta})).$ F. Pellarin has introduced in \cite{PEL} the following entire function on $\mathbb C_{\infty}:$
$$L(t)=\sum_{k\geq 0}\sum_{a\in A_{+,k}}\frac{a(t)}{a}=L(t; z)\mid{z=1}.$$
This function interpolates  the values at  $1$  of  some  abelian $L$-series introduced by David Goss  (see \cite{GOS}, chapter 8) :
$$\forall \zeta\in \overline{\mathbb F_q}, L(t)\mid_{t=\zeta}=L(1, \chi):= \sum_{a\in A_+} \frac{\chi(a)}{a} \in \mathbb C_{\infty}^\times,$$
where $\chi$ is a certain  Dirichlet character attached to $\zeta$ and where $\overline{\mathbb F_q}$ is the algebraic closure of $\mathbb F_q$ in $\mathbb C_{\infty}.$ The function $L(t)$ and its $s$-variable generalization play an important  role in the arithmetic theory of function fields over finite fields (see \cite{ANG&PEL1}, \cite{ANG&PEL2}, \cite{APTR}, \cite{GOS2},  \cite{PER1}, \cite{PER2}, \cite{PER3}). Let's observe that $L(t)\in \mathbb T,$ where $\mathbb T$ is the Tate algebra in the variable $t$ with coefficients in $\mathbb C_{\infty}.$  Now let us fix a $q-1$th root of $-\theta$ in $\mathbb C_{\infty}:\, ^{q-1}\sqrt{-\theta},$ and  let $\omega(t)$ be the Anderson-Thakur special function (see \cite{AND&THA}, proof of Lemma 2.5.4), i. e. :
$$\omega(t)=^{q-1}\sqrt{-\theta} \prod_{j\geq 0}(1-\frac{t}{\theta^{q^j}})^{-1}\in \mathbb T^\times.$$
Let $\tau :\mathbb T\rightarrow \mathbb T$ be the continuous morphism of $\mathbb F_q[t]$-algebras such that $\forall x \in \mathbb C_{\infty}, \tau (x)=x^q.$ We then have the following remarkable formula obtained by F. Pellarin in \cite{PEL}:
\begin{eqnarray} L(t)\omega(t)=log_C(\omega(t))= \frac{\widetilde{\pi}}{\theta-t},\end{eqnarray}
where $log_C= 1+\sum_{j\geq 1} \frac{1}{\prod_{k=0}^{j-1} (\theta- \theta^{q^k})}\tau^j,$ and  $\widetilde {\pi}=^{q-1}\sqrt{-\theta} \theta \prod_{j\geq1}(1-\theta^{1-q^j})^{-1}\in \mathbb C_{\infty}^\times $ is the period of the Carlitz module. Our first goal in this article is to generalize the above formula to the case of  $s$-variable $v$-adic $L$-series  where $v$ is any place of $K$ and $s\geq 1$ is any integer.\par
In this introduction, we will only discuss about the Spiegelungssatz for $L$-series in the case of a prime $P$ of $A$ of degree $d\geq 1$ and $s=1,$ we refer the reader to sections 6 and 7 of this work for the general cases. Our idea is to compare two kinds of $P$-adic $L$-series : one that comes from $P$-adic Lubin-Tate theory and the other that comes from $P$-adic Carlitz theory (the prime $P$ is "viewed" as an infinite place as it will be explained below).\par
\noindent Let's discuss first the case of $P$-adic Lubin-Tate theory.  Let $K_P$ be the $P$-adic completion of $K,$ and  let $\mathbb C_P$ be the completion of a fixed algebraic closure of $K_P.$ Let $v_P:\mathbb C_P \rightarrow \mathbb  Q\cup\{ +\infty\}$ be the valuation on $\mathbb C_P$ normalized such that $v_P(P)=1.$ Let $t$ be an indeterminate over $\mathbb C_P.$ We recall that $(1-\frac{P(t)z^d}{P})L(t;z)$ defines an entire function on $\mathbb C_P^2.$ This $P$-adic function has a trivial zero at $z=1$ (Proposition \ref{Prelproposition1}). Therefore, we set:
$$L_P(t)= \sum_{k\geq 1} \sum_{a\in A_{+,k}, a\not \equiv 0\pmod{P}}k\frac{a(t)}{a}= \frac{d}{dz} (1-\frac{P(t)z^d}{P})L(t;z)\mid_{z=1}.$$  Let $\mathcal A$ be the  completion of $\mathbb C_P[t, \frac{1}{P(t)}]$ for the Gauss $P$-adic valuation on $\mathbb C_P[t, \frac{1}{P(t)}]$ still denoted by $v_P$ (see section 1). Observe that $L_P(t)\in \mathcal A.$  Let $\tau:\mathcal A\rightarrow \mathcal A$ be the continuous morphism of $\mathbb F_q[t, \frac{1}{P(t)}]$-algebras such that $\forall x\in \mathbb C_P,$ $\tau(x)=x^q.$ In section 5, we construct an analogue of Anderson-Thakur's function in this case, it is an element $\omega_P(t)\in \mathcal A$ that satisfies:
$$\tau (\omega_P(t))=(\frac{t-\theta}{t-\zeta_P})\omega_P(t),$$
where $\zeta_P \in \mathbb F_{q^d}\subset \mathbb C_P$ is such that $v_P(\theta-\zeta_P)>0.$  Furthermore, $\omega_P(t)$ interpolates $P$-adically the Gauss-Thakur sums (Theorem \ref{theorem3}). Let $C^{\dag}:A\rightarrow \mathcal A\{ \tau\}$ be the morphism of $\mathbb F_q$-algebras such that:
$$C^{\dag}_{\theta}= (t-\zeta_P)\tau +\theta.$$
Then one can associate to $C^{\dag}$ a logarithm $log_{C^{\dag}}\in \mathcal A\{ \{ \tau\}\},$ i.e. this is the unique element in $\mathcal A\{ \{ \tau\}\}$ such that $log_{C^{\dag}}\equiv 1\pmod{\tau}$ and :
$$log_{C^{\dag}} C^{\dag}_{\theta}= \theta log_{C^{\dag}}.$$
Let $\mathcal A^0=\{ x\in \mathcal A, v_P(x)\geq 0\}$ and $\mathcal A^{00}=\{x \in \mathcal A, v_P(x)>0\}.$ One can easily show that $log_{C^{\dag}}$ converges on $\mathcal A^{00},$ and using similar arguments than that used by Iwasawa for the classical $p$-adic logarithm, we can extend $log_{C^{\dag}}$ into a continuous morphism of $\mathbb F_q[t, \frac{1}{P(t)}]$-modules : $log_{C^{\dag}, P}: \mathcal A^0 \rightarrow \mathcal A.$
We have:
$$log_{C^{\dag}, P}(\omega_P(t))=0.$$
The function $\omega_P(t)$ plays the same role for $L_P(t)$ as that of $\omega(t)$ for $L(t).$ 
Let $C:A\rightarrow A\{ \tau \}$ be the Carlitz module, i.e. $C$ is a morphism of $\mathbb F_q$-algebras such that $C_{\theta}=\tau +\theta.$ Write:
$$C_P=\sum_{k=0}^d (P,k) \tau^k,$$
where $(P,k)\in A,$ $k=0, \ldots, d.$ Then we have an analogous formula as that of formula (1) (see the proof of Theorem \ref{theorem4}):
$$L_P(t)\omega_P(t)=\frac{1}{P}log_{C^{\dag},P}  ( \sum_{k=1}^{d} k(P,k)b_k(t) \omega_P(t)),$$
where $b_0(t)=1$ and for $j\geq 1:$ $b_j(t)=\prod_{k=0}^{j-1} (t-Ê\theta^{q^k}).$\par
\noindent Now, let's focus on the $P$-adic Carlitz theory for $L$-series. We notice that $v_P(\theta-\zeta_P)=1.$ Thus let's work with $\theta-\zeta_P$ instead of $P.$ Let's view $\theta-\zeta_P$ as an infinite prime of $\mathbb F_{q^d}(\theta).$ Let's consider a "twist" of Pellarin's $L$-series  with respect to this infinite prime:
$$\sum_{k\geq 0}\sum_{a\in \mathbb F_{q^d}[\frac{1}{\theta-\zeta_P}]_{+,k}} \frac{a(\frac{1}{t-\zeta_P})}{a(\frac{1}{\theta-\zeta_P})}\in \mathcal A.$$
Let's take its "norm over $\mathbb F_q$":
$$\mathcal L_P(t)= \prod_{j=0}^{d-1}\tau^j(\sum_{k\geq 0}\sum_{a\in \mathbb F_{q^d}[\theta]_{+,k}} \frac{a(\frac{1}{t-\zeta_P})}{a(\frac{1}{\theta-\zeta_P})})\in \mathcal A.$$
Then  $\mathcal L_P(t)$ is convergent  on $\mathbb C_P\setminus\{ \zeta_P, \ldots, \zeta_P^{q^{d-1}}\}.$ One can show that $\omega_P(t)$ plays again an important role for $\mathcal L_P(t)$ (see the proof of Theorem \ref{theorem4}):
$$ \frac{\mathcal L_P(t) \omega_P(t)}{\pi_P}=\frac{P(t)}{b_d(t)},$$
where $\pi_P= \prod_{i=0}^{d-1}\prod_{j\geq 1}(1- (\theta-\zeta_P)^{q^{dj+i}-q^i})^{-1}.$ Therefore, if we combine the formulas for $L_P(t)$ and $\mathcal L_P(t),$ we get a Spiegelungssatz for these $P$-adic $L$-series (Theorem \ref{theorem4}):
$$ \frac{b_d(t)\mathcal L_P(t)}{\pi_P}(  \frac{1}{P}log_{C^{\dag},P}  ( \sum_{k=1}^{d} k(P,k)b_k(t) \omega_P(t)) )                                      =  P(t)L_P(t)                                             ,$$
the above equality taking place in the $P$-adic affinoid algebra $\mathcal A.$\par

 A large part of this work  grew out of an attempt to prove the non-vanishing at one of certain $v$-adic $L$-series (Theorem \ref{theorem1}) without using Vincent Bosser's $v$-adic Baker-Brumer Theorem (see \cite{ANG&TAE1}). In order to do so, one of the main ingredients is to obtain explicit formulas for these $v$-adic values (Theorem \ref{theorem2}). The Spiegelungssatz  for $L$-series can be viewed as  a reinterpretation of such explicit formulas in certain $v$-adic affinoid algebras. The "Anderson-Stark units" (see \cite{APTR}, \cite{APTR2}) and their derivatives  play a key role in this article. Let $\tau= K(t_1, \ldots ,t_s)[[z]]\rightarrow K(t_1, \ldots ,t_s)[[z]]$ be the continuous morphism (for the $z$-adic topology) of $\mathbb F_q(t_1, \ldots, t_s)[[z]]$-algebras given by $\tau (x)=x^q $ for all $x \in K.$ Then there exists $u(t_1, \ldots, t_s; z)\in A[t_1, \ldots, t_s,z]$  such that (\cite{APTR}, Corollary 5.12):
 $$\exp_{s,z}(L(t_1, \ldots, t_s; z))= u(t_1, \ldots, t_s; z),$$
 where $\exp_{s,z}=1+\sum_{j\geq1}\frac{z^j\prod_{i=1}^s\prod_{k=0}^{j-1}(t_i-\theta^{q^k})}{\prod_{k=1}^{j}(\theta^{q^j}-\theta^{q^{j-k}})} \tau ^j.$  Then $u(t_1, \ldots, t_s;z)\mid_{z=1}$ or its derivative (with respect to $z$) at $z=1$  in the case $s\equiv 1\pmod{q-1}$ play a role analogous to  that of Stark units, i.e.  the special values of the $v$-adic $L$-series considered in this paper can be written as linear combinations of $v$-adic logarithms of some elements attached to these Anderson-Stark units and their derivatives (see for example Theorem \ref{theorem2}).  In the last section of this work, using  some properties of Anderson-Stark units and their derivatives, we give some new results on "generic class modules".\par
 In order to explain and motivate the content of the last section of this work, we first recall some facts from classical cyclotomic theory (we refer the reader to \cite{WAS}, chapters 13 and 15). Let $p\geq 3$ be an odd prime number and let $\Delta= {\rm Gal}(\mathbb Q(\mu_p)/\mathbb Q).$ Let $M$ be the projective limit for the norm map  of the $p$-Sylow subgroups of the ideal class groups along the cyclotomic $\mathbb Z_p$-extension of $\mathbb Q(\mu_p): \mathbb Q(\mu_{p^{\infty}}).$ Let $\Gamma={\rm Gal}( \mathbb Q(\mu_{p^{\infty}})/\mathbb Q(\mu_p)),$ and let $\gamma \in \Gamma$ such that:
 $$\forall \zeta \in \mu_{p^{\infty}},\, \gamma (\zeta)=\zeta^{1+p}.$$
 We can identify $\mathbb Z_p[[\Gamma]]$ with $\mathbb Z_p[[T]]$ by sending $\gamma$ to $1+T.$ Therefore $M$ is a $\mathbb Z_p[[T]][\Delta]$-module. Let $\chi \in \widehat {\Delta}:= {\rm Hom}(\Delta, \mathbb Z_p^\times),$ and let $e_{\chi}=\frac{1}{p-1}\sum_{\delta \in \Delta} \chi(\delta) \delta^{-1}\in \mathbb Z_p[\Delta].$ For $\chi \in \widehat{\Delta}, $ we set:
 $$M_{\chi}=e_{\chi}M.$$
 Then $M_{\chi}$ is a finitely generated and torsion $\mathbb Z_p[[T]]$-module. If $\chi$ is odd distinct from the Teichm\"uller character : $\omega_p,$ B. Mazur and A. Wiles proved that the characteristic ideal of the $\mathbb Z_p[[T]]$-module $M_{\chi}$ is generated by some polynomial $P_{\chi}(T)\in \mathbb Z_p[T]$ such that there exists $U_{\chi}(T)\in \mathbb Z_p[[T]]^\times$ with the property (\cite{MAZ}):
 $$\forall y \in \mathbb Z_p, \, P_{\chi}((1+p)^y-1) U_{\chi}((1+p)^y-1)=L_p(y, \omega_p\chi^{-1}),$$
 where $L_p(.,\omega_p\chi^{-1})$ is the $p$-adic $L$-function attached to the even character $\omega_p\chi^{-1}.$ Observe that if Vandiver's Conjecture is true for the prime $p,$ then for $\chi$ odd, $M_{\chi}$ is a cyclic $\mathbb Z_p[[T]]$-module of for $\chi$ even: $M_{\chi}=\{0\}.$ R. Greenberg has conjectured  weaker statements  (see \cite{GRE}):\par
 \noindent {\sl Pseudo-cyclicity Conjecture:} If $\chi$ is odd then there exists an injective  morphism of $\mathbb Z_p[[T]]$-modules between a cyclic $\mathbb Z_p[[T]]$-module and $M_{\chi}$  such that the cokernel of this morphism is finite;\par
 \noindent {\sl Pseudo-nullity Conjecture:} If $\chi$ is even then $M_{\chi}$ is finite.\par
 \noindent These two conjectures are  open.\par
 \noindent Now, we turn ourselves to the function field case. Let $s\geq 1$ be an integer and let $t_1, \ldots,t_s$ be $s$ variables over $ K_{\infty}.$ Let $\mathbb T_s(K_{\infty})$ be the Tate algebra in the variables $t_1, \ldots, t_s$ with coefficients in $K_{\infty}.$ Let $\tau: \mathbb T_s(K_{\infty})\rightarrow \mathbb T_s(K_{\infty})$ be the continuous morphism of $\mathbb F_q[t_1, \ldots, t_s]$-algebras such that $\forall x\in K_{\infty}, \tau (x)=x^q.$ Let $\phi: A\rightarrow A[t_1, \ldots, t_s]\{\tau\}$ be the morphism of $\mathbb F_q$-algebras such that:
 $$\phi_{\theta}= (t_1-\theta)\cdots (t_s-\theta)\tau +\theta.$$
 \noindent Let $\exp_{\phi}$ be the exponential function attached to $\phi,$ i.e. $\exp_{\phi}$ is the unique element in $\mathbb T_s(K_{\infty})\{ \{\tau\}\}$ such that $\exp_{\phi} \equiv 1\pmod{\tau}$ and:
 $$\exp_{\phi} \theta= \phi_{\theta} \exp_{\phi}.$$
 Then inspired by Taelman's work (\cite{TAE1}, \cite{TAE2}, \cite{TAE3}), one can introduce a "generic class module" of level $s$ (see \cite{APTR}):
 $$H_s:= \frac{\mathbb T_s(K_{\infty})                       }{ A[t_1, \ldots, t_s]+\exp_{\phi}(\mathbb T_s(K_{\infty}))                          }.$$
 Observe that $H_s$ is an $A[t_1, \ldots, t_s]$-module via $\phi,$ furthermore it is a finitely generated $\mathbb F_q[t_1, \ldots, t_s]$-module and  a torsion $A[t_1, \ldots,t_s]$-module (we refer the reader to section 8 for more details). The terminology "generic class module" comes from the fact that the evaluations of elements of $\mathbb T_s(K_{\infty})$ on elements of $\overline{\mathbb F_q}^s$ ($\overline{\mathbb F_q}$ being the algebraic closure of $\mathbb F_q$ in $\mathbb C_{\infty}$) induce surjective morphisms from $H_s$ to certain isotypic components of Taelman's class modules. These generic class modules can be viewed as discrete analogues of the Iwasawa modules discussed above, having in mind that the role of $\mathbb Z_p$ is played by $\mathbb F_q[t_1, \ldots, t_s],$ the role of $T$ is played by $\phi_{\theta}.$  For example Mazur-Wiles Theorem has an analogue in this situation, let $s\geq q, $ $s\equiv 1\pmod{q-1}$ (this is the analogue of the condition $\chi$ odd, $\chi \not =\omega_p,$ for Iwasawa modules), then there exists $\mathbb B(t_1, \ldots, t_s)\in A[t_1, \ldots, t_s]$ which is a monic polynomial in $\theta,$ such that the Fitting ideal of the $A[t_1, \ldots, t_s]$-module $H_s$ is generated by $\mathbb B(t_1, \ldots, t_s)$ and (\cite{APTR}, Theorem 7.7):
 $$\mathbb (-1)^{\frac{s-1}{q-1}}\mathbb B(t_1, \ldots, t_s)\frac{\widetilde{\pi}}{\omega(t_1)\cdots \omega(t_s)}= L(t_1, \ldots, t_s),$$
 where $L(t_1, \ldots, t_s)= L(t_1, \ldots, t_s; z)\mid_{z=1}\in \mathbb T_s(K_{\infty}).$ The reader can now easily guess what are the discrete analogues of Greenberg 's Conjectures in our situation. We prove these analogues in section 8:\par
 \noindent {\sl Pseudo-cyclicity:} let $s\geq 2, $ $s\equiv 1\pmod{q-1},$ there exists an injective  morphism of $A[t_1, \ldots, t_s]$-modules between a cyclic $A[t_1, \ldots, t_s]$-module and $H_s$  such that the cokernel of this morphism is a finitely generated and torsion $\mathbb F_q[t_1, \ldots, t_s]$-module (Corollary \ref{ATRcorollary1});\par
 \noindent {\sl Pseudo-nullity:}   let $s\geq 2, $ $s\not \equiv 1\pmod{q-1},$ then $H_s$ is a finitely generated and torsion $\mathbb F_q[t_1, \ldots, t_s]$-modules (Proposition \ref{ATRproposition3} and Theorem \ref{ATRtheorem3}).\par
 \noindent These latter two results have interesting consequences on the behaviour of the isotypic components of Taelman's class modules for the Carlitz module. More precisely, let $s\geq 2$ be an integer, let $\chi$ be a Dirichlet character of type $s$ and conductor $f$ (see section 1). There exist $\zeta_1(\chi), \ldots, \zeta_s(\chi)\in \overline{\mathbb F_q}$ such that:
 $$\sigma (g(\chi))=(\zeta_1(\chi)-\theta)\cdots (\zeta_s(\chi)-\theta) g(\chi),$$
 where $g(\chi)$ is the Gauss-Thakur sum associated to $\chi,$ $\sigma =\tau \otimes 1$ on $O_f\otimes_{\mathbb F_q}\mathbb F_q(\zeta_1(\chi), \ldots, \zeta_s(\chi)),$ where $O_f$ is the ring of integers of the $f$th cyclotomic function field. Let $\exp_C$ be the Carlitz exponential (see \cite{GOS}, chapter 3), and let's consider the $\chi$-isotypic component of Taelman's class module (see \cite{APTR} section 9.3, \cite{ANG&TAE1} and \cite{TAE1})  :
 $$H_{\chi}:=e_{\chi}(\frac{{\rm Frac}(O_f)\otimes_{K}K_{\infty}}{O_F+\exp_C({\rm Frac}(O_f)\otimes_{K}K_{\infty})}\otimes_{\mathbb F_q}\mathbb F_q(\zeta_1(\chi), \ldots, \zeta_s(\chi))),$$
 where $e_{\chi}\in \mathbb F_q(\zeta_1(\chi), \ldots, \zeta_s(\chi))[{\rm Gal}({\rm Frac}(O_f)/K)]$ is the usual idempotent associated to $\chi.$ We prove  that there exists $F(t_1, \ldots,t_s)\in \mathbb F_q[t_1, \ldots, t_s]\setminus \{0\}$  (depending on the module structure of $H_s$) such that for every Dirichlet character $\chi$ of type $s$ verifying $F(\zeta_1(\chi), \ldots , F(\zeta_s(\chi))\not =0,$ then, if $s\equiv 1\pmod{q-1},$ $H_{\chi}$ is a cyclic $A[\zeta_1(\chi), \ldots, \zeta_s(\chi)]$-module generated as a $A[\zeta_1(\chi), \ldots, \zeta_s(\chi)]$-module by $\frac{g(\chi)}{\theta^{\frac{s-q}{q-1}}}$ (Theorem \ref{ATRtheorem2}) and if $s\not \equiv 1\pmod{q-1}$ we have $H_{\chi}=\{ 0\}$  (Corollary \ref{ATRcorollary2}). 

${}$\par
The authors thank David Goss and Federico Pellarin for comments on an earlier version of this work that enable us to improve the content of this article. The authors are indebted to  David Goss for fruitful discussions that led us to find explicit formulas  for the value at one of  $v$-adic Dirichlet $L$-series for finite places $v$ of $K.$\par


\section{Preliminaries}
\subsection{Background on entire functions}${}$\par
Let $F$ be a local field of characteristic $p$ containing a finite field $\mathbb F_q$ having $q$ elements. Let $O_F$ be the valuation ring of $F,$ and let $\pi_F$ be a generator of the maximal ideal of $O_F.$ Let $v:F\rightarrow \mathbb Z\cup\{ +\infty\}$ be the discrete valuation on $F$  such that $v(\pi_F)=1.$  Let $\mathbb C_F$ be the completion of a fixed algebraic closure $\overline{F}$ of $F,$ and let $v:\mathbb C_F\rightarrow \mathbb Q\cup \{ +\infty\}$ be the valuation on $\mathbb C_F$ such that $v(\pi_F)=1.$ Let $\tau :\mathbb C_F\rightarrow \mathbb C_F, x\mapsto x^q.$\par
Let $n\geq 1$ be an integer,  let $t_1, \ldots,t_n$ be $n$ indeterminates over $\mathbb C_F.$ We set $\mathbb C_F[\underline{t}_n]= \mathbb C_F[t_1, \ldots, t_n],$ and if $\underline{i}=(i_1,\ldots, i_n)\in \mathbb N^n,$ we set $\underline{t}_n^{\underline{i}}= t_1^{i_1}\cdots t_n^{i_n}$ and $\mid \underline{i}\mid = i_n +\cdots+i_n.$  Let $f\in \mathbb C_F[\underline{t}_n],$ $f=\sum_{\underline{i}\in \mathbb N^n}f_{\underline{i}}\underline{t}^{\underline{i}},$ $f_{\underline{i}}\in \mathbb C_F.$ We set :
$$v(f)={\rm Inf}\{ v( f_{\underline{i}}), \underline{i}\in \mathbb N^n\}.$$
Then $v$ induces a valuation on $\mathbb C_F(\underline{t}_n)$ and we call it the $v$-adic  Gauss valuation on $\mathbb C_F(\underline{t}_n).$\par

\noindent Let $\mathbb T_n$ be the Tate algebra in the variables $t_1, \ldots,t_n$ with coefficients in $\mathbb C_F$ (see \cite{PUT}, chapter 3), i.e. $\mathbb T_n$ is the completion of $\mathbb C_F[\underline{t}_n]$ for the $v$-adic Gauss valuation.  Then $\mathbb T_n$ is the set of elements $f=\sum_{\underline{i}\in \mathbb N^n}f_{\underline{i}}\underline{t}_n^{\underline{i}} \in \mathbb C_F[[\underline{t}_n]], $ $f_{\underline{i}}\in \mathbb C_F$ for all $\underline{i}\in \mathbb N^n,$ such that :
$$\lim_{\mid \underline{i}\mid \rightarrow +\infty}v(f_{\underline{i}})=+\infty .$$
More generally, for $r\in \mathbb Q,$ we define $\mathbb T_{n,r}$ to be the set of elements $f=\sum_{\underline{i}\in \mathbb N^n}f_{\underline{i}}\underline{t}_n^{\underline{i}} \in \mathbb C_F[[\underline{t}_n]], $ $f_{\underline{i}}\in \mathbb C_F$ for all $\underline{i}\in \mathbb N^n,$ such that :
$$\lim_{\mid \underline{i}\mid \rightarrow +\infty}v(f_{\underline{i}})+r\mid \underline{i}\mid =+\infty .$$
Thus $\mathbb T_{n,0}=\mathbb T_n,$ and $\mathbb T_{n,r}$ is  the set of elements $f\in \mathbb C_F[[\underline{t}_n]]$ that converge on $\{ \underline{\zeta}\in \mathbb C_F^n, v(\underline{\zeta})\geq r\}.$ Observe that if $r,r'\in \mathbb Q,$ if $r'\leq r$ then $\mathbb T_{n,r'}\subset \mathbb T_{n,r}.$    For $f=\sum_{\underline{i}\in \mathbb N^n}f_{\underline{i}}\underline{t}_n^{\underline{i}} \in \mathbb T_{n,r}, $ $f_{\underline{i}}\in \mathbb C_F,$ we set :
$$v_{r}(f)={\rm Inf}\{ v(f_{\underline{i}})+r\mid \underline{i} \mid, \underline{i}\in \mathbb N^n\},$$
and $v=v_{0}.$ For $r\in \mathbb Q,$ we set:
$$\mathbb T_{n,r}(F)= \mathbb T_{n,r}\cap F[[\underline{t}_n]].$$
Let $r\in \mathbb Q$ and let $\zeta\in \mathbb C_F^*$ such that $v(\zeta)=-r.$ Then the map :
$$ \psi : \mathbb T_{n,r}\rightarrow \mathbb T_n, f(\underline{t}_n)\mapsto f(t_1\zeta, \ldots , t_n\zeta),$$
is an isomorphism of $\mathbb C_F$-algebras such that :
$$v(\psi(f))=v_{r}(f).$$
Let $\tau :\mathbb C_F[[\underline{t}_n]]\rightarrow \mathbb C_F[[\underline{t}_n]]$ be the isomorphism of $\mathbb F_q[[\underline{t}_n]]$-algebras such that :
$$\tau(\sum_{\underline{i}\in \mathbb N^n}f_{\underline{i}}\underline{t}_n^{\underline{i}})=
\sum_{\underline{i}\in \mathbb N^n}f_{\underline{i}}^q\underline{t}_n^{\underline{i}}, \, f_{\underline{i}}\in \mathbb C_F.$$
Observe that,  for $r\in \mathbb Q,$ we have :
$$\tau(\mathbb T_{n,r})=\mathbb T_{n,qr}.$$
We set :
$$\mathbb E_n=\cap_{r\in \mathbb Q}\mathbb T_{n,r}\subset \mathbb T_n,$$
and :
$$\mathbb E_n(F)=\mathbb E_n\cap F[[\underline{t}_n]].$$
Then $\mathbb E_n$ is the set of elements $f=\sum_{\underline{i}\in \mathbb N^n}f_{\underline{i}}\underline{t}_n^{\underline{i}} \in \mathbb C_F[[\underline{t}_n]], $ $f_{\underline{i}}\in \mathbb C_F$ for all $\underline{i}\in \mathbb N^n,$ such that :
$$\lim_{\mid \underline{i}\mid \rightarrow +\infty}\frac{v(f_{\underline{i}})}{\mid \underline{i}\mid}=+\infty .$$
We call the $\mathbb C_F$-algebra $\mathbb E_n$ {\em the algebra of  entire functions} on $\mathbb C_F^n.$ Observe that $\mathbb C_{\infty}[\underline{t}_n]\subsetneq \mathbb E_n \subsetneq \mathbb T_n,$ and  $\tau(\mathbb E_n)=\mathbb E_n.$
Set $\mathbb E_0=\mathbb C_F.$ Let $n\geq 1.$ Observe that $\mathbb E_n$ is the set of elements 
$f=\sum_{j\geq 0} f_j t_n^j\in \mathbb E_{n-1}[[t_n]], f_j \in \mathbb E_{n-1},$ such that :
$$\lim_{j\rightarrow +\infty} \frac{v(f_j)}{j}=+\infty.$$
We recall the following basic result :

\begin{lemma}\label{ATlemma1}${}$\par
\noindent Let $S\subset \mathbb C_F$ be a bounded  infinite set, i.e. there exists $M\in \mathbb R$ such that $\forall x\in S, v(x)\geq M.$ Let $f\in \mathbb E_n$ that vanishes on $S^n,$ then $f=0.$\par
\end{lemma}
\begin{proof}  We prove the assertion by induction on $n\geq 1.$  The result is true for $n=1$ by \cite{GOS}, proposition 2.13. Now, let $n\geq 2$ and assume that the result is proved for $n-1.$ Let $f\in \mathbb E_n.$ Write :
$$f=\sum_{j\geq 0} f_j t_n^j , \, f_j \in \mathbb E_{n-1}, j\geq 0.$$
Let $\zeta_1, \ldots, \zeta_{n-1}\in S^{n-1}$ and set $g(t_n)=f(\zeta_1, \ldots, \zeta_{n-1}, t_n)\in \mathbb C_F[[t_n]].$ Then $g$ is an entire function which  vanishes on $S.$ Thus $g=0.$ Therefore  $f=0$ by the induction hypothesis. \end{proof}
\begin{lemma}\label{ATlemma2}
Let r $\in \mathbb Q, r\leq 0.$ Let $F=\sum_{i\geq0}F_i\tau ^i \in \mathbb T_{n,r}\{ \{ \tau\}\}, $ $F_i \in \mathbb T_{n,r},$ such that:
$$\lim_{i\rightarrow +\infty} \frac{v_{r}(F_i)}{q^i}=+\infty.$$
Then $F(\mathbb T_{n,r})\subset \mathbb T_{n,r}.$
\end{lemma}
\begin{proof} Observe that $v_{r}$ is a valuation on $\mathbb T_{n,r}$ and that $\mathbb T_{n,r}$ is complete with respect to the topology defined by this valuation. Let $f\in \mathbb T_{n,r}$ then:
$$\forall k\geq 0, v_{r}(\tau^k(f))\geq q^kv_{r}(f).$$
Thus:
$$\lim_{k\rightarrow +\infty} v_{r}(F_k\tau^k(f))= +\infty.$$
\end{proof}
Let $T\in \mathbb C_F$ such that $v(T)=-1.$ We set $A_T=\mathbb F_q[T],$ $D_{0,T}=1$ and for $i\geq 1:$ $D_{i,T}= (T^{q^i}-T)D_{i-1,T}^q.$
\begin{proposition}\label{ATproposition1}
For all $i\geq 0, $ let $a_i\in \mathbb C_F[\underline {t}_n]$ such that there exists a constant $c\leq 0$ with $\forall i\geq 0, $ $v(a_i)\geq cq^i,$ and there exists a constant $c'\geq 0$ such that the total degree (in $\underline{t}_n$) of $a_i$ is bounded by $q^ic'.$ Let $F=\sum_{i\geq 0} \frac{a_i}{D_{i,T}} \tau^i \in \mathbb E_n\{\{ \tau \}\}.$  \par
\noindent i) $\forall r\in \mathbb Q, r\leq 0,$ $F(\mathbb T_{n,r})\subset \mathbb T_{n,r}.$ In particular $F(\mathbb E_n)\subset \mathbb E_n.$\par
\noindent ii) Let $r<0.$ $\{ f\in  \mathbb T_{n,r}, F(f)\in \mathbb E_n\}= \mathbb E_n.$\end{proposition}
\begin{proof}  For $i\geq 0,$ we set :
$$F_i=\frac{a_i}{D_{i,T}}.$$ Let $r\leq 0.$ We have:
$$v_{r}(F_i) \geq iq^i +cq^i +rq^ic'.$$ 
Therefore :
$$\lim_{i\rightarrow +\infty} \frac{v_{r}(F_i)}{q^i}=+\infty.$$
Thus, by Lemma \ref{ATlemma2}, for all $r\leq 0,$ $F(\mathbb T_{s,r})\subset \mathbb T_{s,r}.$\par
\noindent Let $r\in \mathbb Q, r<0.$  Let $f\in \mathbb T_{n,r} $  such that $F(f)\in \mathbb E_n.$ We have to prove that  $f\in \mathbb E_n.$ Assume that this is not the case. Then we could find $r'<r<0$ such that $f\in \mathbb T_{n,r'}\setminus \mathbb T_{n, qr'}.$ But observe that :
$$\sum_{i\geq 1}F_i\tau^i(f) \in \mathbb T_{n,qr'}.$$
Thus, we would get :
$$f= F(f)-\sum_{i\geq 1}F_i\tau^i(f) \in \mathbb T_{n,qr'},$$
which is absurd.
\end{proof}
\noindent We finish this section by the following Lemma which will be used in the sequel :
\begin{lemma}\label{ATlemma3}
Let $a\in A_T[\underline{t}_n]$ which is monic as a polynomial in $T.$ Then, there exists $r\in \mathbb Q, r<0,$ such that $\frac{1}{a}\in \mathbb T_{n,r}.$
\end{lemma}
\begin{proof}   Write :
$$a=T^k +b,$$
where $k\in \mathbb N, $ and $b\in A_T[\underline{t}_n]$ and $v(b)>-k.$ Let $d\geq 0$ be the total degree of $b$ in $\underline{t}_n.$ Let $r\in \mathbb Q, r<0.$ Then for $m\geq 0:$
$$v_{r}(\frac{b^m}{T^{km}})\geq m+rmd.$$
Let's choose now $\frac{-1}{d}< r< 0.$ Then :
$$\lim_{m\rightarrow +\infty} v_{r}(\frac{b^m}{T^{km}})=+\infty.$$
Therefore $\frac{1}{a}= \sum_{m\geq 0} (-1)^m \frac{b^m}{T ^{k(m+1)}}$ converges in $\mathbb T_{n,r}.$
\end{proof}

\subsection{ Gauss-Thakur sums}${}$\par
Let $\mathbb F_q$ be a finite field having $q$ elements, $\theta$ be an indeterminate over $\mathbb F_q,$ $A=\mathbb F_q[\theta],$ $A_{+}$ be the set of monic elements in $A,$ $A_{+,k}$ be the set of elements in $A_{+}$ of degree $k\geq 0,$ $K=\mathbb F_q(\theta).$ Let $\infty$ be the unique place of $K$ which is a pole of $\theta.$ Let $K_{\infty}= \mathbb F_q((\frac{1}{\theta}))$ and let $\mathbb C_{\infty}$ be the completion of a fixed algebraic closure of $K_{\infty}.$ Let $v_{\infty}: \mathbb C_{\infty}\rightarrow  \mathbb Q\cup \{ +\infty\}$ be the valuation on $\mathbb C_{\infty}$ normalized such that $v_{\infty} (\theta) =-1.$ Let $\overline{K}$ be the algebraic closure of $K$ in $\mathbb C_{\infty},$ and  let $\overline{\mathbb F_q}$ be the algebraic closure of $\mathbb F_q$ in $\mathbb C_{\infty}.$\par 
\noindent Let $\exp_C $ be the Carlitz exponential (see \cite{GOS}, chapter 3). For every $f\in A\setminus \{0\},$ we set $\lambda_f=\exp_C(\frac{\widetilde{\pi}}{f})\in \overline{K},$ where :
$$\widetilde{\pi} = ^{q-1}\sqrt{-\theta}\,  \theta \prod_{j\geq 1}(1- \theta^{1- q^j})^{-1} \in \mathbb C_{\infty}.$$
We set $\Delta_f= {\rm Gal}(K(\lambda_f)/K).$ We refer the reader to \cite{ROS}, chapter 12, for the basic properties of the abelian extension $K(\lambda_f)/K.$ For $a\in A,$ $a$ relatively prime to $f,
$ we denote by $\sigma_a$ the element in $\Delta_f$ such that :
$$\sigma_a(\lambda_f)= C_a(\lambda_f),$$
where $C$ is the Carlitz module (see \cite{GOS}, chapter 3).\par

Let $P$ be a prime of $A$ (i.e. a monic irreducible element in $A$) of degree $d\geq 1.$ Let $K_P$ be the $P$-adic completion of $K$ and $A_P$ be the valuation ring of $K_P.$ Let $\mathbb C_P$ be the completion of a fixed algebraic closure of $K_P$ and we fix once and for all a $K$-embedding of $\overline{K}$ in $ \mathbb C_P.$ Let $v_P$ be the $P$-adic valuation on $\mathbb C_P$ normalized such that $v_P(P)=1.$ Let $O_{\mathbb C_P}=\{ x\in \mathbb C_P, v_P(x)\geq 0\}.$ \par

\noindent Let $sgn= sgn_P : O_{\mathbb C_P}\rightarrow \overline{\mathbb F_q}\subset O_{\mathbb C_P}$ be the morphism of $\overline{\mathbb F_q}$-algebras  such that :
$$\forall x\in O_{\mathbb C_P}, v_P(sgn(x)-x)>0.$$
If $x\in O_{\mathbb C_P}^\times,$ we write $x=sgn(x)<x>,$ where $v_P(1-<x>)>0.$
We set  $\zeta_P= sgn (\theta).$ Observe that :
$$P(\zeta_P)=0.$$
A Dirichlet character, $\chi,$  is a morphism :
$$\chi: (\frac{A}{bA})^{\times}\rightarrow \overline{\mathbb F_q}\subset O_{\mathbb C_P},$$
for some squarefree $b\in A_{+}.$ Let $\mathbb F_q(\chi) $ be the field obtained by adjoining to $\mathbb F_q$ the values of $\chi.$ If we view $\chi$ as a primitive character of conductor $f,$ we extend $\chi$ into a map $\chi :A\rightarrow\mathbb F_q(\chi),$ by sending $a$ to $0$ if $a$ is not prime to $f,$ and $\chi( a) = \chi (a \pmod{f}) $ otherwise.  We say that $\chi$ is of type $1,$ if $\chi :A\rightarrow \mathbb F_q(\chi)$ is a ring homomorphism.
Observe that if $\chi$ is of type $1,$ then its conductor is a prime of $A.$ Let $\chi_P$ be  the Dirichlet character of type $1$ and conductor $P$ such that $\chi_P(\theta)=\zeta_P.$\par

\noindent Let $f$ be a prime of $A$  and let $\chi$ be a Dirichlet character of type $1$ and conductor $f,$  recall that the Gauss-Thakur sum associated to $\chi$ is defined by (see \cite{THA}) :
$$\forall i\geq 0, g(\chi)= -\sum_{ \sigma \in \Delta_f} \chi^{-1}(\sigma) \sigma (\lambda_f),$$
where  we recall that $\Delta_f={\rm Gal}(K(\lambda_f)/K).$ Now let $\chi$ be a primitive Dirichlet character of conductor $P_1\cdots P_r,$ where $P_1, \ldots, P_r$ are distinct primes of $A.$ For $i=1, \ldots, r,$ let $\chi_i$ be a Dirichlet character of type $1$ and conductor $P_i,$ then :
$$\chi= \chi_{1}^{N_1}\cdots \chi_{r}^{N_r},$$ where $0\leq N_j\leq q^{deg P_j}-2.$ Write, for $j=1, \ldots , r,$ $N_j =\sum_{i=0}^{deg_{P_j}-1} n_{i,j} q^i , $ $n_{i,j}\in \{ 0\ldots, q-1\}.$ The Gauss-Thakur sum associated to $\chi$ is (see for example \cite{ANG&PEL2}) :
$$g(\chi)=\prod_{j=1}^r \prod_{i=0}^{deg P_j-1} g(\chi_{j}^{q^i})^{n_{i,j}}.$$
The type of $\chi$ (see also \cite{APTR}) is defined to be $\sum_{j=1}^r\sum_{i= 0}^{deg P_j-1} n_{i, j}.$ \par
\noindent If $\chi$  is a Dirichlet character of conductor $b,$ we have (see \cite{ANG&PEL2}):
$$g(\chi)g(\chi^{-1})= (-1)^{deg b}b.$$

\subsection{ $L$-functions}${}$\par
In this section, we recall the functional identities for certain $s$-variable $L$-series obtained in \cite{ANG&PEL2}, \cite{PEL}. Let $s\geq 0$ be an integer and  $Z, t,y_1, \ldots,y_s$ be $s+2$ variables over $\mathbb F_q.$ We set for $d\geq 1$:
$$\mathcal L(\mathbb F_{q^d};  y_1, \ldots, y_s; Z)= \sum_{ a\in \mathbb F_{q^d}[\theta]_{+}}\frac{a(y_1) \cdots a(y_s)}{a(\frac{1}{Z})} \in \mathbb F_{q^d}[y_1, \ldots, y_s][[Z]]^\times.$$
We also set :
$$\omega(\mathbb F_{q^d}; t; Z)= \prod_{j\geq 0} (1- tZ^{q^{dj}})^{-1} \in \mathbb F_{q^d}[t][[Z]]^\times ,$$
and :
$$\pi(\mathbb F_{q^d}; Z)=  \prod_{j\geq 1} (1- Z^{q^{dj}-1})^{-1} \in \mathbb F_{q^d}[[Z]]^\times .$$
\begin{proposition}\label{Preltheorem1}${}$\par
\noindent Let $d\in \mathbb N\setminus\{0\}.$\par
\noindent 1) For $s=1,$ set $y=y_1.$ Then:
$$\frac{\mathcal L(\mathbb F_{q^d};  y; Z)\, \omega(\mathbb F_{q^d}; y; Z)}{\pi(\mathbb F_{q^d}; Z)}= \frac{1}{1-yZ}.$$
\noindent 2) Let $s\equiv 1\pmod{q^d-1}, $  $s\geq q^d.$ Then there exists $B(\mathbb F_{q^d}; y_1, \ldots, y_s; Z)\in \mathbb F_{q^d}[y_1, \ldots, y_s, Z]\cap (1+Z\mathbb F_{q^d}[y_1, \ldots, y_s][[Z]] )$ such that :
$$\frac{\mathcal L(\mathbb F_{q^d};  y_1, \ldots, y_s; Z)\, \omega(\mathbb F_{q^d}; y_1; Z)\cdots \omega(\mathbb F_{q^d}; y_s; Z)}{\pi(\mathbb F_{q^d}; Z)}= B(\mathbb F_{q^d}; y_1, \ldots, y_s; Z) .$$
\end{proposition}
\begin{proof}  It is enough to prove the result for $d=1.$ Let $t, t_1, \cdots, t_s$ be $s+1$ indeterminates over $\mathbb C_{\infty},$ then  set :
$$L(t_1, \ldots , t_s)= \mathcal L(\mathbb F_q;  t_1, \ldots, t_s; \frac{1}{\theta}) \in K_{\infty}[[t_1, \ldots, t_s]].$$
This is the $L$-function introduced in \cite{PEL} and studied in  \cite{ANG&PEL1},  \cite{ANG&PEL2}, \cite{APTR}, \cite{GOS2}, \cite{PER1}, \cite{PER2}, \cite{PER3}. Recall that $L(t_1, \ldots, t_s)$ defines an entire function on $\mathbb C_{\infty}^s$ (\cite{ANG&PEL2} Proposition 6). Let's set :
$$\pi_{\infty}= \pi(\mathbb F_q; \frac{1}{\theta}) \in K_{\infty}^\times,$$
and :
$$\omega_{\infty} (t) =\omega(\mathbb F_q; t; \frac{1}{\theta})\in K_{\infty}[[t]].$$
Then :
$$\widetilde{\pi} = \lambda_{\theta}\, \theta\,  \pi_{\infty},$$
and the Anderson-Thakur function (see for example \cite{ANG&PEL1}), $\omega(t),$ is equal to :
$$\omega(t) =\lambda_{\theta}\, \omega_{\infty}(t).$$
Then assertion 1) is a consequence of \cite{PEL} Theorem 1, and assertion 2) is a consequence of \cite{ANG&PEL2} Theorem 1 and  \cite{APTR} section 7.1.1. .
\end{proof}
\noindent Recall that we have (see for example  \cite{PEL}):
$$\omega(t) =\sum_{i\geq 0} \lambda_{\theta^{i+1}} t^i.$$

\noindent We will need the following Lemma in the sequel :
\begin{lemma}\label{Prellemma1}
Let $\mathbb T$ be the Tate algebra in the variable $t$ with coefficients in $\mathbb C_{\infty},$ (see \cite{PUT}, chapter 3),  we have the following equality in $\mathbb T:$
$$    \omega(t) =\sum_{j\geq 0}\lambda_{(\theta+1)^{j+1}} (t+1)^j           .$$
\end{lemma}
\begin{proof} Set :
$$F(t)=\sum_{j\geq 0}\lambda_{(\theta+1)^{j+1}} (t+1)^j     \in \mathbb T^\times.$$
Let $\tau :\mathbb T\rightarrow \mathbb T$ be the continuous morphism of $\mathbb F_q[t]$-algebras such that :
$$\forall x\in \mathbb C_{\infty}, \tau (x) =x^q.$$
Recall that $$\exp_C=\sum_{i\geq 0} \frac{1}{D_i} \tau ^i,$$
where $D_0=1,$ and for $i\geq 1: D_i=(\theta^{q^i}-\theta) D_{i-1}^q.$ Then (see for example \cite{PEL}):
$$\omega(t) =\exp_C (\frac{\widetilde{\pi}}{\theta-t}).$$
Thus :
$$\omega(t)= \exp_C(\frac{\widetilde{\pi}}{(\theta+1)-(t+1)})= F(t).$$
\end{proof}

\subsection{ $P$-adic interpolation}${}$\par
Let $P$ be a prime of $A$ of degree $d\geq 1,$ and let $z, t_1, \ldots, t_s$ be $s+1$ indeterminates over $\mathbb C_P.$ We set :
$$L_P(t_1, \ldots, t_s; z) =\sum_{k\geq 0} \sum_{a\in A_{+,k}, a\not \equiv 0\pmod{P}} \frac{a(t_1)\cdots a(t_s)}{a} z^k \in A_P[t_1, \ldots, t_s][[z]].$$
\begin{theorem}\label{Preltheorem2}
$L_P(t_1, \ldots, t_s; z)$ defines an entire function on $\mathbb C_P^{s+1}.$
\end{theorem}
\begin{proof} The proof of the above Theorem is essentially a consequence of  \cite{ABBL} section 3.3 and the proof of \cite{ANG&PEL2} Proposition 6. Set :
$$V_k(t_1, \ldots, t_s)= \sum_{a\in A_{+,k}, a\not \equiv 0\pmod{P}} \frac{a(t_1)\cdots a(t_s)}{a} ,$$
and :
$$S_k(t_1, \dots, t_s)=\sum_{a\in A_{+,k}} a(t_1)\cdots a(t_s) .  $$
Let $v_P$ be the $P$-adic Gauss valuation on $K_P(t_1, \ldots t_s).$ By \cite{ANG&PEL2} Lemma 4, if $s< k(q-1),$ we have $S_k(t_1, \ldots, t_s)=0.$
Now, let $n\geq 1$ be an integer, then :
$$\forall a \in A\setminus PA, v_P(\frac{1}{a} -a^{q^{dn}(q^d-1)-1}) \geq q^{dn}.$$
Thus:
$$v_P(V_k(t_1, \ldots t_s)-\sum_{a\in A_{+,k}} a(t_1)\cdots a(t_s) a^{q^{dn}(q^d-1)-1})\geq q^{dn}.$$
But , if $k(q-1)> s+nd(q-1)+q^d-2,$ we have :
$$\sum_{a\in A_{+,k}} a(t_1)\cdots a(t_s) a^{q^{dn}(q^d-1)-1}=0.$$
Thus if $k(q-1)> s+nd(q-1)+q^d-2,$ we have :
$$v_P(V_k(t_1, \ldots, t_s))\geq q^{nd}.$$
\end{proof}
\begin{proposition}\label{Prelproposition1}
Let $s\geq 1.$ We have $L_P(t_1, \ldots, t_s;z)\mid_{z=1}= 0$ if and only if $s\equiv 1\pmod{q-1}.$
\end{proposition}
\begin{proof}  Let's set :
$$F(t_1, \ldots , t_s)= L_P(t_1, \ldots, t_s;z)\mid_{z=1}.$$
Then by Theorem \ref{Preltheorem2}, $F(t_1, \ldots, t_s)$ defines an entire function on $\mathbb C_P^{s}.$  Let $i_1, \ldots, i_s \in \mathbb N.$
Then, if $s\not \equiv 1\pmod{q-1},$ we have :
$$F(\theta^{q^{i_1}}, \ldots, \theta^{q^{i_s}})= (1-P^{q^{i_1}+\cdots +q^{i_s}-1})\beta (q^{i_1}+\cdots + q^{i_s}-1) \not =0,$$
where $\beta (i)$ denotes the $i$th Bernoulli-Goss number (see for example  \cite{GOS3}, \cite{IRE&SMA}). Now, if $s\equiv 1\pmod{q-1},$ we have :
$$F(\theta^{q^{i_1}}, \ldots, \theta^{q^{i_s}})= (1-P^{q^{i_1}+\cdots +q^{i_s}-1})(\sum_{k\geq 0}\sum_{a\in A_{+,k}} a^{q^{i_1}+\cdots + q^{i_s}-1}).$$
But, since $q^{i_1}+\cdots + q^{i_s}\equiv 1\pmod{q-1},$ we get (\cite{GOS}, section 8.13.) :
$$F(\theta^{q^{i_1}}, \ldots, \theta^{q^{i_s}})=0.$$
But $F(t_1, \ldots, t_s)$ is an entire function, thus by Lemma \ref{ATlemma1},  if $s\equiv 1\pmod{q-1},$ we get that $F(t_1, \ldots, t_s)=0.$
\end{proof}
If $s\not \equiv 1\pmod{q-1},$ we set :
$$L_P(t_1, \ldots, t_s)=L_P(t_1, \ldots, t_s, z)\mid_{z=1},$$
and if $s\equiv 1\pmod{q-1} :$
$$L_P(t_1, \ldots, t_s)= \frac{d}{dz}L_P(t_1, \ldots, t_s, z)\mid_{z=1}.$$
Note that, by Theorem \ref{Preltheorem2}, $L_P(t_1, \ldots, t_s)$ defines an entire function on $\mathbb C_P^{s}.$

\begin{proposition}\label{Prelproposition2} Let $s\geq 1$ be an integer. Let $v_P$ be the Gauss $P$-adic valuation on $A_P[[t_1, \ldots, t_s]].$ Then :
$$v_P(L_P(t_1, \ldots, t_s))=0.$$
\end{proposition} 
\begin{proof} We  prove the assertion for $s\not \equiv 1\pmod{q-1}.$ 
Since $L_P(t_1, \dots, t_s)$ is an entire function, we can evaluate it at $t_s= \theta$, we get 
$$ L_P(t_1, \dots, t_s)\mid_{t_s=\theta} = \sum_{k\geq 0} \sum_{a\in A_{+,k}, a\not \equiv 0\pmod{P} }a(t_1)\cdots a(t_s)  \in \mathbb F_{q}[t_1, \ldots, t_s]$$
But $L_P(t_1, \dots, t_s)\mid_{t_1=\theta, \dots, t_s=\theta}  = (1-P^{s-1})\beta(s-1)$ where $\beta(s-1)\neq 0$ is the $s-1$th Bernoulli-Goss number (see \cite{IRE&SMA}). 
Thus, $ L_P(t_1, \dots, t_s)\mid_{t_s=\theta} \in \mathbb F_{q}[t_1, \ldots, t_s] \backslash \{0\}$, so that $v_P(L_P(t_1, \ldots, t_s))=0.$

The proof of the Theorem is similar for $s\geq 2, $ $s\equiv 1\pmod{q-1}.$ It remains to treat the case $s=1.$ Recall that by Theorem \ref{Preltheorem2}, $L_P(t)$ is an entire function. Since $\beta(q-1)=1,$ we have :
$$L_P(\theta^q)=\sum_{k\geq 1}\sum_{a\in A_{+,k}, a\not \equiv 0\pmod{P}} a^{q-1} = P^{q-1}-1.$$
 \end{proof}


\section{Anderson's $P$-adic logarithm}${}$\par
Let $P$ be a prime of $A$ of degree $d\geq 1.$ In this section, we recall the construction of the $P$-adic logarithm of Anderson (\cite{AND}, proposition 11) and we extend it to  $O_{\mathbb C_P}.$\par
Let $\tau : \mathbb C_P \rightarrow \mathbb C_P$ be the isomorphism of $\mathbb F_q$-algebras given by :
$$\forall x\in \mathbb C_P, \tau (x)=x^q.$$
Let $C: A\rightarrow A\{ \tau \}$ be the Carlitz module, i.e. $C$ is the morphism of $\mathbb F_q$-algebras given by :
$$C_{\theta}= \tau +\theta.$$
For $a\in A\setminus\{ 0\},$ let  $\Lambda_a\subset \overline{K}$ be the set of $a$-torsion points for the Carlitz module and set $\Lambda =\cup_{a\in A\setminus\{0\}} \Lambda_a.$\par
\begin{lemma}
\label{lemma1}${}$\par
\noindent i) Let $n\geq 1,$ then $K_P(\Lambda_{P^n-1})=K_P(\mathbb F_{q^{dn}}).$ Furthermore, let $\varphi\in {\rm Gal}(K_P(\mathbb F_{q^{dn}})/K_P)$ be the Frobenius morphism then :
$$\forall \lambda \in \Lambda_{P^n-1}, \varphi (\lambda)=C_P(\lambda).$$
ii) $K_P(\Lambda)$ is the maximal abelian extension of $K_P.$
\end{lemma}
\begin{proof} This Lemma can be deduced from the work of David Hayes (\cite{HAY}). For the convenience of the reader, we  give a proof of it.\par
\noindent i) We observe that :
$$C_{P^n-1}\equiv \tau^{dn}-1\pmod{PA\{\tau \}}.$$
Thus, by Hensel's Lemma :  $K_P(\Lambda_{P^n-1})=K_P(\mathbb F_{q^{dn}}).$ Furthermore $\varphi (\Lambda_{P^n-1})=\Lambda_{P^n-1}.$ Thus, the second assertion comes from the fact that if $\lambda, \lambda' \in \Lambda_{P^n-1},$ $\lambda\not = \lambda',$ then $v_P(\lambda-\lambda')=0.$\par
\noindent ii) Observe that $C_P$ is a Lubin-Tate polynomial. Thus, assertion ii)  is a consequence of i) and Lubin-Tate theory (see \cite{NEU}, chapter V).
\end{proof}
\noindent Recall that the Carlitz exponential is defined by :
$$\exp_C =\sum_{i\geq 0} \frac{1}{D_i } \tau^i \in K\{ \{ \tau\}\},$$
where $D_0=1$ and for $i\geq 1: D_i= (\theta^{q ^i}-\theta) D_{i-1}^q.$
We have the following equality in $K\{ \{ \tau \}\} :$
$$ \exp_C \theta = C_{\theta} \exp_C.$$
Furthermore, the Carlitz logarithm is defined by :
$$log_C= \sum_{i\geq 0} \frac{1}{\ell_i} \tau ^i \in K\{\{ \tau \}\},$$
where $\ell_0=1,$ and for $i\geq 1:  \ell_i =(\theta -\theta^{q^i})\ell_{i-1}.$ We also have the following equality in $K\{ \{ \tau \}\}:$
$$ \exp_C log_C= log_C \exp_C =1.$$
\begin{lemma}\label{lemma2}${}$\par
\noindent i) $\exp_C$ converges on $\{ x\in \mathbb C_P, v_P(x)>\frac{1}{q^d-1}\}.$  Furthermore, for $x\in \mathbb C_P, $ $v_P(x)>\frac{1}{q^d-1},$ we have $v_P(\exp_C(x))=v_P(x).$\par
\noindent ii) $log_C$ converges on $\frak M_{\mathbb C_P}=\{ x\in \mathbb C_P, v_P(x)>0\}.$ Furthermore, for $x\in \mathbb C_P, $ $v_P(x)>\frac{1}{q^d-1},$ we have $v_P(log_C(x)-x)>v_P(x).$

\end{lemma}
\begin{proof}
Let $[x]$ denotes the integer part of $x\in \mathbb R.$ The Lemma  comes from the following facts :
$$\forall i\geq 0, v_P(\ell_i)=[\frac{i}{d}],$$
$$\forall i\geq 0, v_P(D_i)=\frac{q^i-q^{i-[\frac{i}{d}]d }}{q^d-1}.$$
\end{proof}
\noindent We set $\Lambda'=\cup_{n\geq 1} \Lambda_{P^n-1}$ and $\Lambda''=\cup_{n\geq 1}\Lambda_{P^n}\subset \frak M_{\mathbb C_P}.$ We have a direct sum of $A$-modules $\Lambda= \Lambda'\oplus \Lambda ''.$ Observe that $O_{\mathbb C_P}$ has  a structure of  $A$-module via the Carlitz module (denoted by $C(O_{\mathbb C_P})$). Thus Lemma \ref{lemma1} implies that we have a direct sum of $A$-modules :
$$C(O_{\mathbb C_P})= \Lambda'\oplus C(\frak M_{\mathbb C_P}).$$
Furthermore one can show that $C(\frak M_{\mathbb C_P})$ is an $A_P$-module but we will not need this fact.
\begin{lemma}\label{lemma3} ${\rm Tor}_AC(\frak M_{\mathbb C_P})=\Lambda ''$ and :
$${\rm Ker}(log_C: C(\frak M_{\mathbb C_P})\rightarrow \mathbb C_P) =\Lambda ''.$$
\end{lemma}
\begin{proof} The first assertion is obvious by our previous discussions. Let $x\in \frak M_{\mathbb C_P}.$ Then there exists an integer $n\geq 0$ such that $v_P(C_{P^n}(x))>\frac{1}{q^d -1}.$ Thus, by Lemma \ref{lemma2}, we have  :
$$v_P(log_C(x)) =v_P(\frac{1}{P^n}log_C(C_{P^n}(x)) = v_P(C_{P^n} (x))-n.$$
The Lemma follows easily.
\end{proof}
\noindent Now, let $x\in C(O_{\mathbb C_P}),$ then there exists an integer $n\geq 1,$ such that $v_P(C_{P^n-1}(x))>0.$ We define the $P$-adic Anderson's logarithm of $x$ (this is an analogue construction as that of the Iwasawa logarithm) as follows:
$$log_{C,P}(x)=\frac{1}{P^n-1}log_C(C_{P^n-1}(x)).$$
Note that $log_{C,P}(x)$ is well-defined. We have immediately by Lemma \ref{lemma3}:
\begin{lemma}\label{lemma4} The map $log_{C,P}: C(O_{\mathbb C_P})\rightarrow \mathbb C_P$ is a morphism of $A$-modules and its kernel is exactly $\Lambda.$

\end{lemma}
\section{Non-vanishing of David Goss $P$-adic Dirichlet $L$-functions}${}$\par
Let $P$ be a prime of $A$ of degree $d.$ Let $\chi$ be a Dirichlet character of type $1$ and conductor $f.$ Recall that :
$$g(\chi)= -\sum_{\sigma \in \Delta_f}\chi^{-1}(\sigma)\sigma (\lambda_f).$$
Let $\sigma =\tau\otimes 1 $ on $A[\lambda_f]\otimes_{\mathbb F_q}\mathbb F_q(\chi),$ then :
$$\sigma (g(\chi))= (\chi (\theta)-\theta) g(\chi).$$
Now, let $\chi$ be a Dirichlet character of type $s\geq 0$ and conductor $f,$  then the above equation and the definition of $g(\chi)$ imply that there exist $\zeta_1, \cdots, \zeta_s \in \overline{\mathbb F_q}$ such that:
$$\sigma (g(\chi))= (\zeta_1-\theta)\cdots (\zeta_s-\theta) g(\chi),$$
where $\sigma =\tau \otimes 1 $ on $A[\lambda_f ]\otimes_{\mathbb F_q} \mathbb F_q(\zeta_1, \ldots, \zeta_s).$ Furthermore, one sees easily that if $a\in A$ is relatively prime to $f,$ then :
$$\chi(a)= \chi (\sigma_a) = a(\zeta_1)\cdots a(\zeta_s).$$
By Theorem \ref{Preltheorem2}, we can set  :
$$L_P(1, \chi)= L_P(t_1, \ldots, t_s) \mid_{t_1= \zeta_1, \ldots, t_s= \zeta_s} \in A_P[\chi].$$
Observe that is $s\not \equiv 1 \pmod{q-1},$ we have :
$$L_P(1, \chi)=\sum_{k\geq 0} \sum_{a\in A_{+,k} , a\not \equiv 0\pmod{P}}\frac{\chi (a)}{a} ,$$
where $\chi$ is viewed as defined modulo its conductor and $\chi (a)=0$ if $a$ is not relatively prime to $f,$ and if $s\equiv 1\pmod{q-1} :$
$$L_P(1, \chi)=\sum_{k\geq 0} \sum_{a\in A_{+,k} , a\not \equiv 0\pmod{P}}k\frac{\chi (a)}{a}  .$$
Our aim in this section is to prove the following Theorem :
\begin{theorem}
\label{theorem1} Let $\chi$ be a Dirichlet character of type $s$ and assume $q\geq 3$ if $s=0.$ Then :
$$L_P(1, \chi)\not =0.$$
\end{theorem}
Note that the case of characters of conductor $P$ and $s\not \equiv 1\pmod{q-1}$ is treated in \cite{ANG&TAE1} and the   non-vanishing result  uses Vincent Bosser's $P$-adic Baker Brumer Theorem (see the appendix of \cite{ANG&TAE1}). In what follows, we propose a new approach which does not use the $P$-adic Baker-Brumer Theorem. The proof of Theorem \ref{theorem1} is quite  technical and we first need some preliminary results.\par
\noindent Let $s\geq 0$ be an integer. We will now work in $K[t_1, \ldots, t_s][[z]].$ Let 
$$\tau : K[t_1, \ldots, t_s][[z]]\rightarrow K[t_1, \ldots, t_s][[z]]$$
 be the continuous morphism (for the $z$-adic topology) of $\mathbb F_q[t_1,\ldots, t_s][[z]]$-algebras such that $$\forall x\in K, \tau(x)=x^q.$$  Set :
$$log_{s,z}=\sum_{k\geq 0} \frac{b_i(t_1)\cdots b_i(t_s)}{\ell_k}z^k \tau ^k,$$
where for $j=1, \ldots, s,$ $b_0(t_j)=1$ and for $i\geq 1:$ $b_i(t_j)=\prod_{k=0}^{i-1} (t_j-\theta^{q ^k}).$ 
Then, Corollary 5.12  in \cite{APTR} implies that there exists $u(t_1, \ldots, t_s;z) \in A[t_1, \ldots, t_s,z]$ such that we have the following equality in $K[t_1, \ldots, t_s][[z]]:$
$$\sum_{k\geq 0}\sum_{a\in A_{+,k}}\frac{a(t_1)\cdots a(t_s)}{a} z^k = log_{s,z} u(t_1, \ldots, t_s;z).$$
We have the following properties for the "Anderson-Stark units" $u(t_1, \ldots,t_s;z):$
\begin{lemma}\label{lemma5}${}$\par
\noindent i)  If $0\leq s\leq q-1,$ $u(t_1, \ldots, t_s;z)=1.$\par
\noindent ii) If $s\geq q, $ $s\equiv 1 \pmod{q-1},$ $u(t_1, \ldots, t_s;z)\mid_{z=1}= 0.$\par
\noindent iii) If $s\geq 1,$ we have :
$$u(t_1, \ldots, t_s;z)\mid_{t_s=\theta}= \sum_{k\geq 0}\sum_{a\in A_{+,k}}a(t_1)\cdots a(t_{s-1}) z^k.$$
In particular if $s\geq q, s\equiv 1\pmod{q-1}:$
$$\frac{d}{dz}(u(t_1, \ldots, t_s;z))\mid_{z=1}\not =0.$$
\end{lemma}
\begin{proof} Let's set :
$$\exp_{s,z} =\sum_{i \geq 0} \frac{b_i(t_1)\cdots b_i(t_s)}{D_i}z^i \tau ^i.$$
We work in the Tate algebra in the variables $t_1, \ldots, t_s, z$ with coefficients in $K_{\infty}.$
Set :
$$L_{s,z}= \sum_{k\geq 0} \sum_{a\in A_{+,k}}\frac{a(t_1)\cdots a(t_s)}{a} z^k.$$
Then by   \cite{APTR},  Corollary 5.12 , we have   :
$$u(t_1, \ldots, t_s; z)= \exp_{s,z}(L_{s,z}).$$
Furthermore, by \cite{APTR} Lemma 3.9 and \cite{ANG&PEL2} Proposition 6,  this is an equality between entire functions in the variables $t_1, \ldots, t_s,z.$
Thus we get the first assertion of iii) . Fo the second assertion of iii), observe that :
$$\frac{d}{dz}(u(t_1, \ldots, t_s;z))\mid_{z=1, t_1=\theta, \ldots , t_s=\theta }= \beta (s-1)\not =0,$$
where $\beta (i)$ denotes the $i$th Bernoulli-Goss number.\par
\noindent For ii), it is a consequence of \cite{APTR}, section 7. 
For i), observe that :
$$v_{\infty} (\frac{b_i(t_1)\cdots b_i(t_s)}{D_i}z^i)=i q^i-s\frac{q^i-1}{q-1}.$$
Therefore, if $0\leq s\leq q-1,$ we get :
$$v_{\infty} (\exp_{s,z}(L_{s,z}) -1)\geq 1.$$
Since $u(t_1, \ldots, t_s,z)\in A[t_1, \ldots , t_s, z],$ we get i).
\end{proof}
\begin{proposition}\label{proposition1} ${}$\par
\noindent Let $s\geq q.$ Write $s=q+r+\ell (q-1),$ $\ell \in \mathbb N,$ $r\in \{ 0, \ldots, q-2\}.$Then $u(t_1, \ldots, t_s;z)$ viewed as a polynomial in $\theta$ is of degree $s(\frac{q^{\ell}-1}{q-1})-\ell q^{\ell}$ if $r=0$ and $s(\frac{q^{\ell +1}-1}{q-1})-(\ell +1)q^{\ell +1}$ if $r\not =0.$ Furthermore the leading coefficient of $u(t_1,\ldots, t_s;z)$ viewed as a polynomial in $\theta $ is :
$$ {\rm if }\, r=0, (-1)^{\ell}z^{\ell}(1-z).$$
$${\rm if}\, r\not =0, (-1)^{s(\ell +1)}z^{\ell +1}     .$$
\end{proposition}
\begin{proof} We prove the result for $r=0,$ the proof of the Proposition being similar in the remaining cases. We will use freely the results  in \cite{APTR}. Let $s\geq q,$ $s\equiv 1\pmod{q-1}.$
Let $\phi$ be the Drinfeld module defined on the Tate algebra in the variables $t_1, \ldots, t_s, z$ with coefficients in $K_{\infty}$  such that $\phi_{\theta}= z(t_1-\theta)\cdots (t_s-\theta) \tau +\theta.$
Recall that :
$$u:=u(t_1, \ldots, t_s;z)=\exp_{\phi}L,$$
where :
$$L=\sum_{k\geq 0} \sum_{a\in A_{+,k}} \frac{a(t_1)\cdots a(t_s)}{a} z^k.$$
In particular $v_{\infty}(L-1)\geq 1.$ Recall that :
$$\exp_{\phi}=\sum_{i\geq 0} \frac{b_i(t_1)\cdots b_i(t_s)z^i}{D_i}\tau ^i.$$
We say that an  element $f$ in $\mathbb F_q(t_1, \ldots, t_s,z)((\frac{1}{\theta}))$ is " monic" if there exists  an integer $m\in \mathbb Z$ such that :
$$v_{\infty}(f- \theta^m )> v_{\infty}(f).$$
For $i\geq 0$ observe that $\tau^i(L)$ and $\frac{b_i(t_1)\cdots b_i(t_s)}{D_i}$ are almost monic. Recall that :
$$\forall i\geq 0, v_{\infty}(\frac{b_i(t_1)\cdots b_i(t_s)z^i}{D_i}\tau^i (L))=i q^i-s(\frac{q^i-1}{q-1}) .$$
Set $\alpha_i= (-1)^{s i}\frac{b_i(t_1)\cdots b_i(t_s)}{D_i}\tau^i (L),$  then $\alpha_i$ is  monic. We have:
$$u=\sum_{i\geq 0}(-1)^{s i}\alpha_i z ^i.$$
For $i\geq 0,$ we have :
$$((i+1) q^{i+1}-s(\frac{q^{i+1}-1}{q-1})) -(i q^i-s(\frac{q^i-1}{q-1}))=q^i(q+i(q-1)-s)    .$$
Write $s=q+\ell (q-1), $ $\ell \in \mathbb N.$ We get :
$$v_{\infty}(u-((-1)^{s\ell }\alpha_{\ell} z^{\ell }+ (-1)^{s(\ell+1)}\alpha_{\ell +1}z^{\ell +1}))>v_{\infty}(u).$$
But $\alpha_{\ell}$ and $\alpha_{\ell +1}$ are monic, thus:
$$v_{\infty}(u-(z^{\ell}(1-z)(-1)^{s \ell}\theta^{\ell q^{\ell}-s(\frac{q^{\ell}-1}{q-1})})>v_{\infty}(u).$$
This proves the  assertion. 
\end{proof}
\noindent Let $\chi$ be a Dirichlet character of type $s$ and conductor $f.$ Recall that there exists $\zeta_1, \ldots, \zeta_s\in \overline{\mathbb F_q}$ such that :
$$\sigma (g(\chi))=(\zeta_1-\theta)\cdots  (\zeta_s-\theta) g(\chi),$$
where $\sigma =\tau \otimes 1$ on $A[\lambda_f]\otimes_{\mathbb F_q} \mathbb F_q(\zeta_1, \ldots, \zeta_s).$
We set :
$$u_{\chi}(z) = g(\chi)u(t_1, \ldots, t_s; z)\mid_{t_1=\zeta_1, \ldots , t_s=\zeta_s}\in g(\chi)\mathbb  F_q(\zeta_1, \ldots , \zeta_s)[z].$$
Let $[\chi]= \{ \chi^{q^i}, i\geq 0\}.$ Observe that for $\psi \in [\chi],$ $\psi$ is of type $s$ and conductor $f.$ We set :
$$u_{[\chi]}(z)=\sum_{\psi \in [\chi]} u_{\psi}(z) \in A[\lambda_f][z],$$
and :
$$u^{(1)}_{[\chi]}(z)= \frac{d}{dz} u_{[\chi]}(z).$$

\begin{theorem}\label{theorem2}
Let $\chi$ be a Dirichlet character of type $s$ and conductor $f.$\par
\noindent 1) If $s\not \equiv 1\pmod{q-1},$ then :
$$L_P(1,\chi)g(\chi) =(1-P^{-1}\chi (P))\frac{1}{\mid \Delta_f\mid }\sum_{\mu \in \Delta_f}\chi^{-1}(\mu)log_{C,P} (\mu(u_{[\chi]}(1)))                           .$$
2) If $s\geq q, $ $s\equiv 1\pmod{q-1},$  then :
$$   L_P(1, \chi) g(\chi)=(1-P^{-1}\chi (P))\frac{1}{\mid \Delta_f\mid }\sum_{\mu \in \Delta_f}\chi^{-1}(\mu)log_{C,P} (\mu(u^{(1)}_{[\chi]}(1)))                       .$$
3) If $s=1,$ we have :
$$   L_P(1,\chi)g(\chi )=(1-P^{-1}\chi (P))\frac{1}{\mid \Delta_f\mid }\frac{1}{\theta- \chi(\theta)}\sum_{\mu \in \Delta_f}\chi^{-1}(\mu) log_{C,P}( \mu (\lambda_f^q))                        .$$
\end{theorem}
\begin{proof}  We first  work in $K[t_1, \ldots, t_s][[z]].$ Let $\tau : K[t_1, \ldots, t_s][[z]]\rightarrow K[t_1, \ldots, t_s][[z]]$ be the continuous morphism (for the $z$-adic topology)  of $\mathbb F_q[t_1,\ldots, t_s][[z]]$-algebras such that $\forall x\in K, \tau(x)=x^q.$  Set :
$$log_{s,z}=\sum_{k\geq 0} \frac{b_i(t_1)\cdots b_i(t_s)}{\ell_k}z^k \tau ^k.$$
Then recall  that  we have the following equality in $K[t_1, \ldots, t_s][[z]]:$
$$\sum_{k\geq 0}\sum_{a\in A_{+,k}}\frac{a(t_1)\cdots a(t_s)}{a} z^k = log_{s,z} u(t_1, \ldots, t_s;z).$$
Let $\sigma =\tau \otimes 1$ on $K[\lambda_f] \otimes_{\mathbb F_q}\mathbb F_q (\zeta_1, \ldots, \zeta_s).$  And set :
$$  log_{\sigma, z}=\sum_{k\geq 0}\frac{1}{\ell_k} z^k \sigma^k.                                                            $$
If we specialize $t_i$ in $\zeta_i$ and if we multiply the formula by $g(\chi),$ we have the following  equality in $K[\lambda_f]\otimes_{\mathbb F_q}\mathbb F_q(\zeta_1, \ldots, \zeta_s))[[z]]:$
$$\sum_{k\geq 0}\sum_{a\in A_{+,k}}\frac{\chi(a)}{a} g(\chi)z^k = log_{\sigma,z} u_{\chi}(z).$$
Recall that $\Delta_f={\rm Gal}(K(\lambda_f)/K),$ and  that for $\mu \in \Delta_f,$ we have :
$$\mu (g(\chi))=\chi(\mu)g(\chi).$$
Now, set :
$$log_{C,z}=\sum_{k\geq 0}\frac{1}{\ell_k} z^k \tau^k.$$
We get the following equality in $K[\lambda_f][[z]]$ ($\tau$ acts trivially on $z$):
$$\sum_{\psi \in  [\chi]} \sum_{k\geq 0}\sum_{a\in A_{+,k}}\frac{\psi(a)}{a} g(\psi)z^k= log_{C,z} u_{[\chi]}(z).$$
By the properties of the Gauss-Thakur sums, we deduce that :
$$\mid \Delta_f\mid \sum_{k\geq 0}\sum_{a\in A_{+,k}}\frac{\chi(a)}{a} g(\chi)z^k= \sum_{\mu \in \Delta_f}\chi^{-1}(\mu)log_{C,z} \mu(u_{[\chi]}(z)).$$
Therefore:
$$\mid \Delta_f\mid  \, g(\chi)  \sum_{k\geq 0}\sum_{a\in A_{+,k}}\frac{\chi(a)}{a} z^k =\sum_{\mu \in \Delta_f}\chi^{-1}(\mu)log_{C,z} \mu(u_{[\chi]}(z)).$$
And finally :
$$g(\chi ) \sum_{k\geq 0}\sum_{\substack{a\in A_{+,k} \\ a\not \equiv 0\pmod{P}}}\frac{\chi(a)}{a} z^k = (1-P^{-1}\chi (P)z^d)\frac{1}{\mid \Delta_f\mid }\sum_{\mu \in \Delta_f}\chi^{-1}(\mu)log_{C,z} \mu(u_{[\chi]}(z)).$$ 
Recall that :
$$\theta log_{C, z}= log_{C,z} (z\tau +\theta).$$
Now, since $u_{[\chi]}(1)$ can be a $P$-adic unit, using the functional equation of $log_{C,z}$ if necessary, we get if $\chi$ is even (i.e. $s\not \equiv 1\pmod{q-1}$) :
$$L_P(1,\chi)g(\chi) =(1-P^{-1}\chi (P))\frac{1}{\mid \Delta_f\mid }\sum_{\mu \in \Delta_f}\chi^{-1}(\mu)log_{C,P} (\mu(u_{[\chi]}(1))).$$ 
Observe that $u_{[\chi]}(1) \in \exp_C(K(\lambda_f)\otimes_KK_{\infty})$, i.e. $u_{[\chi]}(1)$ is a Taelman unit (see \cite{TAE1}, \cite{TAE2}). More precisely, we have the following equality in $K(\lambda_f)\otimes_KK_{\infty}:$
$$\exp_C(\sum_{\psi \in [\chi]} L(1, \psi) g(\psi))= u_{[\chi]}(1),$$
where $\tau : K(\lambda_f)\otimes_KK_{\infty}\rightarrow K(\lambda_f)\otimes_KK_{\infty}, \alpha \mapsto \alpha^q.$ 
Now, let's treat the case where  $\chi$ is odd (i.e. $s\equiv 1\pmod{q-1}$). If $s\geq q,$ by Lemma \ref{lemma5} we have  $u_{[\chi]}(1)=0.$  Using again the same technique, we get in this case:
$$L_P(1, \chi) g(\chi)=(1-P^{-1}\chi (P))\frac{1}{\mid \Delta_f\mid }\sum_{\mu \in \Delta_f}\chi^{-1}(\mu)log_{C,P} (\mu(u^{(1)}_{[\chi]}(1))).$$ 
Thus, it remains to treat the case $s=1$ . First, by \cite{THA} Proposition I   and Lemma \ref{lemma5}:
$$u_{[\chi]}(z)=\lambda_f.$$
Therefore, in this case :
$$g(\chi) \sum_{k\geq 0}\sum_{\substack{a\in A_{+,k}\\ a\not \equiv 0\pmod{P}}}\frac{\chi(a)}{a} z^k = (1-P^{-1}\chi (P)z^d)\frac{1}{\mid \Delta_f\mid }\sum_{\mu \in \Delta_f}\chi^{-1}(\mu)log_{C,z} \mu(\lambda_f).$$ 
Now, select $\lambda \in \mathbb F_q$ such that $f$ is relatively prime to $\theta +\lambda.$ Then :
$$\sum_{\mu \in \Delta_f}\chi^{-1}(\mu)log_{C,z} \mu(\lambda_f) = \chi ^{-1}(\theta+\lambda)
\sum_{\mu \in \Delta_f}\chi^{-1}(\mu)log_{C,z}C _{\theta+\lambda}(\mu(\lambda_f)).$$
Thus $g(\chi) (\chi(\theta+\lambda)-(\theta +\lambda))\sum_{k\geq 0}\sum_{a\in A_{+,k}, a\not \equiv 0\pmod{P}}\frac{\chi(a)}{a} z^k$ is equal to:
$$ (1-P^{-1}\chi (P)z^d)\frac{1}{\mid \Delta_f\mid }\sum_{\mu \in \Delta_f}\chi^{-1}(\mu)(log_{C,z} C_{\theta+\lambda }(\mu(\lambda_f))- (\theta +\lambda) log_{C,z} \mu(\lambda_f)).$$ 
But $s=1,$ thus  $\chi (\theta +\lambda)= \chi (\theta) +\lambda.$ Thus, 
$$g(\chi)(\chi(\theta)-\theta)\sum_{k\geq 0}\sum_{a\in A_{+,k}, a\not \equiv 0\pmod{P}}\frac{\chi(a)}{a} z^k$$
 is equal to :
$$ (1-P^{-1}\chi (P)z^d)\frac{1}{\mid \Delta_f\mid }\sum_{\mu \in \Delta_f}\chi^{-1}(\mu)(log_{C,z} C_{\theta }(\mu(\lambda_f))- \theta  log_{C,z} \mu(\lambda_f)).$$ 
Set :
$$log^{(1)}_{C,z}= \sum_{k\geq 0}\frac{k}{\ell_k} z^{k-1} \tau^k.$$
Then  :
$$log^{(1)}_{C,z} (z\tau +\theta) -\theta log^{(1)}_{C,z}= -log_{C,z} \tau.$$
Therefore, for $s=1,$ we get :
$$g(\chi) L_P(1,\chi)=(1-P^{-1}\chi (P))\frac{1}{\mid \Delta_f\mid }\frac{1}{\theta- \chi(\theta)}\sum_{\mu \in \Delta_f}\chi^{-1}(\mu) log_{C,P}( \mu (\lambda_f^q)) .$$
\end{proof}

${}$\par
\noindent {\sl Proof of Theorem \ref{theorem1} :}
We first treat the case where $\chi$ is even (i.e. $s\not \equiv 1\pmod{q-1}$). Note that  by Proposition \ref{proposition1}, we have :
$$u_{[\chi]}(1)\not =0.$$
Since $\chi$ is even this implies that $u_{[\chi]}(1)$ is not a torsion point for the Carlitz module, in particular :
 $$log_{C,P}(u_{[\chi]}(1))\not =0.$$
 But, by the proof of Theorem \ref{theorem2}, we have  :
 $$log_{C,P}(u_{[\chi]}(1))=\sum_{\psi \in [\chi]}(1-\frac{\psi(P)}{P})^{-1}L_P(1, \psi) g(\psi) .$$
 This implies that there exists $\psi \in [\chi]$ such that :
 $$L_P(1, \psi)\not =0.$$
 Thus, we have to prove that if $L_P(1, \chi) \not =0,$ then $L_P(1, \chi^{\frac{1}{q}})\not =0.$ We are going to prove it by performing a change of variables.\par 
 \noindent 
 \noindent Set ${K^{(q)}}=\mathbb F_q(\theta^q), $ observe that $K/K^{(q)} $ is totally ramified at every place of $K^{(q)}.$ Let ${C^{(q)}}$ be the Carlitz module for $A^{(q)}:=A^q=\mathbb F_q[\theta^q], $ i.e. ${C^{(q)}}_{\theta^q}= \tau +\theta ^q.$ Let $\widetilde{\lambda_f}=\exp_{{C^{(q)}}} (\frac{\widetilde{\pi }^{(q)}}{f(\theta^q)}).$ Then :
 $$K(\lambda_f) =K(\lambda_{f(\theta^q)}^{(q)}).$$
 Furthermore, the restriction map induces an isomorphism of groups :
 $$\Delta_f \simeq \Delta_{f(\theta^q)}^{(q)}.$$
 This implies that :
 $$\forall \sigma \in \Delta_f, \chi (\sigma) =\chi(\sigma^{(q)}).$$
 Now $L_P(1, \chi)\not =0$ implies that $L_P^{(q)}(1, \chi)\not =0.$
 But :
 $$L_P^{(q)}(1, \chi)= (L_P(1, \chi^{\frac{1}{q}}))^q.$$
 Thus :
 $$L_P(1,\chi^{\frac{1}{q}})\not =0.$$
Now, we turn to the case where $\chi$ is odd of type $\geq q.$ Again, by Proposition \ref{proposition1}, we get :
$$u^{(1)}_{[\chi]}(1) \not =0.$$
Since $\chi$ is not of type $1,$ by \cite{THA} Proposition I,  $u^{(1)}_{[\chi]}(1)$ is not a torsion point,  in particular:
 $$log_{C,P}(u^{(1)}_{[\chi]}(1))\not =0.$$
Now by the proof of Theorem \ref{theorem2}:
$$log_{C,P}(u^{(1)}_{[\chi]}(1))= \sum_{\psi \in [\chi]} (1-\frac{\psi (P)}{P})^{-1}L_P(1, \psi) g(\psi) .$$
We can conclude as in the even case.
It remains to treat the case $s=1.$ Observe that $\lambda_f^q$ is not a torsion point for the Carlitz module. Furthermore by Theorem \ref{theorem2}, we have :
$$     log_{C,P}(\lambda_f^q)= \sum_{\psi\in [\chi]}(\theta-\psi(\theta))(1-\frac{\psi(P)}{P})^{-1} L_P(1, \psi)g(\psi).$$                                                      
Again, we can conclude as in the even case. This achieves the proof of Theorem \ref{theorem1}.


\section{The special function $\omega_P$}
Let $P$ be a prime of $A$ of degree $d\geq 1.$  In this section, we extend \cite{ANG&PEL1}, Theorem 2.9, to the $P$-adic case. Let's consider $\mathbb C_P[t,\frac{1}{P(t)}]$ equipped with the usual Gauss $P$-adic valuation $v_P$ and let $\mathcal A$ be the completion of $\mathbb C_P[t,\frac{1}{P(t)}]$ for $v_P.$ Instead of working in $P$-adic Tate algebras as in \cite{APTR}, we will work in the affinoid algebra $\mathcal A$  (see \cite{PUT} chapter 3). Let $\tau : \mathcal A\rightarrow \mathcal A$ be the continuous morphism of $\mathbb F_q[t, \frac{1}{P(t)}]$-algebras such that :
$$\forall x \in \mathbb C_P, \tau (x)=x^q.$$
For $j=0, \ldots, d-1,$ we set  (see section 2.2.  for the definition of $\omega(\mathbb F_{q^d}; y; Z)$):
$$\omega_{j,P}(t) =\tau^j(\omega(\mathbb F_{q^d}; \frac{1}{t-\zeta_P}; {\theta-\zeta_P}))\in \mathbb F_{q^d}[t, \frac{1}{P(t)}][[P]]^\times\subset \mathcal A.$$
We also set :
$$\omega_P(t) =\prod_{j=0}^{d-1} \omega_{j,P}(t).$$
Observe that :
$$\omega_P(t) =\prod_{j\geq 0}(1-\frac{(\theta-\zeta_P)^{q^j}}{t-\zeta_P^{q^j}})^{-1}.$$
And also :
$$\tau (\omega_P(t))= (\frac{t-\theta}{t-\zeta_P})\, \omega_P(t).$$
We fix $\gamma \in \mathbb C_P$ such that :
$$\gamma ^{q^d-1}= \zeta_P-\theta .$$
\begin{theorem}\label{theorem3}
Let $\chi$ be a Dirichlet character of type $1.$ Let $f$ be the conductor of $\chi.$\par
\noindent 1) If $f\not = P.$ Then :
$$ g(\chi)= sgn(g(\chi)) (\omega_P\mid_{t=\chi(\theta)}).$$
Furthermore :
$$sgn(g(\chi))^{q^{deg f}-1}= \prod_{j=0}^{deg f-1} (\chi(\theta) -\zeta_P^{q^j}).$$
In particular :
$$sgn(g(\chi^{-1}))g(\chi)= (-1)^{\rm deg f}f(\zeta_P)(\omega_P\mid_{t=\chi(\theta)}).$$
2) If $f=P,$ write $\chi =\chi_P^{q^i},$ $i\in \{ 0, \ldots, d-1\}.$ Then:
$$   g(\chi)=  \gamma ^{q^i} sgn (\frac{g(\chi)}{\gamma ^{q ^i}}){P'(\zeta_P)^{-q^i}} (\prod_{j=0, j\not = i}^{d-1} \omega_{j,P})\mid_{t= \chi (\theta)}                                        .$$
Furthermore :
$$sgn (\frac{g(\chi)}{\gamma ^{q ^i}})^{q^d -1}=P'(\zeta_P^{q^i}).$$
\end{theorem}
\begin{proof}${}$\par
\noindent 1)  Let $\zeta=\chi(\theta).$
We have :
$$\tau(g(\chi))= (\zeta^q -\theta)g(\chi^q).$$
Thus:
$$\tau^{deg f}(g(\chi))= b_{deg f}(\zeta) g(\chi).$$
Since  $f\not = P,$  $g(\chi)$ is a $P$-adic unit and this gives the second assertion of 1). Let $n= (q^{deg f}-1)d\, deg f$ then  (see for example \cite{ROS}, Theorem 12.10.) :
$$\mathbb F_{q^d}(\zeta, sgn(\lambda_f))\subset \mathbb F_{q^n}.$$  We get for any $k\geq 0 :$
$$g(\chi)^{q^{kn}}= \tau^{kn}(g(\chi))= b_{nk}(\zeta) g(\chi).$$
Now, observe that :
$$\lim_{k\rightarrow +\infty} g(\chi)^{q^{kn}}= sgn(g(\chi)).$$
Now we turn our attention to $b_{nk}(t):$
$$b_{nk}(t)=\prod_{j=0}^{nk-1} (t-\theta^{q^k}).$$
Since $d$ divides $n,$ we can write the above product :
$$b_{nk}(t)=\prod_{i=0}^{nk/d-1}\prod_{j=0}^{d-1}(t-\zeta_P^{q^j}- (\theta^{q^j}-\zeta_P^{q^j})^{q^{di}}).$$
Thus:
$$b_{nk}(t) =P(t)^{nk/d}\prod_{i=0}^{nk/d-1}\prod_{j=0}^{d-1}(1- \frac{(\theta^{q^j}-\zeta_P^{q^j})^{q^{di}}}{t-\zeta_P^{q^j}}).$$
Observe that  $P(\zeta) \in \mathbb F_{q^{deg f}}^\times.$ By our choice of $n:$
$$P(t)^{nk/d}\mid_{t=\zeta}=1.$$ Now, for $j=0, \ldots, d-1,$ we have  :
$$\omega_{j,P}(t)=\prod_{k\geq 0}(1- \frac{(\theta^{q^j}-\zeta_P^{q^j})^{q^{dk}}}{t-\zeta_P^{q^j}})^{-1}.$$
We get the  first formula of 1). For the last assertion of 1), use the fact that :
$$g(\chi)g(\chi^{-1})= (-1)^{deg f} f.$$
\noindent 2) Now let $\chi =\chi_P^{q^i}, $ $i\in \{ 0, \ldots, d-1\}.$ Then as in the proof of part 1):
$$\tau^d(g(\chi))= b_d(\chi(\theta))g(\chi).$$
But:
$$b_d(\chi(\theta))=P'(\zeta_P^{q^i})( \zeta_P-\theta)^{q^i}\prod_{j=0, j\not = i}^{d-1}(1-\frac{(\theta-\zeta_P)^{q^j}}{\zeta_P^{q^i}-\zeta_P^{q^j}}).$$
Thus :
$$\tau^d(\frac{g(\chi)}{\gamma ^{q ^i}})= P'(\zeta_P^{q^i})(\prod_{j=0, j\not = i}^{d-1}(1-\frac{(\theta-\zeta_P)^{q^j}}{\zeta_P^{q^i}-\zeta_P^{q^j}}))\frac{g(\chi)}{\gamma ^{q ^i}}.$$
We conclude as in the proof of part 1).
\end{proof}

\section{The Spiegelungssatz}
Let $P$ be a prime of $A$ of degree $d\geq 1.$ Our aim in this section is to show that the $P$-adic $L$-functions $L_P(t_1,\ldots, t_s), s\geq 1,$ verify a particular kind of functional identity which can be seen as a Spiegelungssatz for $P$-adic $L$-series.

\subsection{The case $s=1$}${}$\par
Let $\varepsilon_P =1$ if $P=\theta$ and zero otherwise. Let $\mathcal B\subset \mathbb C_P[[t+\varepsilon_P]]$ be the set of power series in $t+\varepsilon_P$ with bounded coefficients and let $v_P$ be the Gauss $P$-adic valuation on $\mathcal B.$ Let $\tau :\mathcal B\rightarrow \mathcal B$ be the continuous morphism of $\mathbb F_q[[t+\varepsilon_P]]$-algebras such that:
$$\forall x\in \mathbb C_P, \tau(x)=x^q.$$
We set $\mathcal B^0= O_{\mathbb C_P}[[t+\varepsilon_P]]\subset \mathcal B.$ We extend the function $sgn : O_{\mathbb C_P} \rightarrow \overline{\mathbb F_q}$ into a continuous morphism of $\overline{\mathbb F_q}[[t+\varepsilon_P]]$-algebras :
$$sgn : \mathcal B^0 \rightarrow \overline{\mathbb F_q}[[t+\varepsilon_P]].$$
We set :
$$\mathcal B^{00} ={\rm Ker}\,  sgn .$$
Observe also that :
$$\mathcal A\subset \mathcal B.$$
We set :
$$\mu(t) =\sum_{j\geq0}sgn(\lambda_{(\theta+\varepsilon_P)^{j+1}})(t+\varepsilon_P)^j.$$
Note that :
$$\mu(t) \in \mathcal B^0.$$
We have :
$$\tau (\mu(t)) = (t-\zeta_P) \mu(t).$$
Therefore :
$$\tau (\omega_P(t)\mu(t))=(t-\theta)\omega_P(t)\mu(t),$$
and :
$$\omega_P(t)\mu(t)\in \mathcal B^0
.$$
\begin{lemma}\label{lemma6}
$\mu(t)$ is transcendental over ${\rm Frac}(\mathcal A).$ In particular $\omega_P(t)\mu(t)$ is transcendental over ${\rm Frac}(\mathcal A).$
\end{lemma}
\begin{proof} Observe that, for all $n\geq 0,$ we have :
$$\tau^{dn} (\mu(t))= P(t)^n \mu(t).$$
This easily implies that $\mu(t)$ is transcendental over $\overline{\mathbb F_q}(t).$
Set $ \mathcal A^{0}= \mathcal B^{0}\cap \mathcal A,$  then we have :
$$sgn(\mathcal A^{0}) = \overline{\mathbb F_q}[t,\frac{1}{P(t)}].$$
If $\mu(t)$ was algebraic over ${\rm Frac}(\mathcal A),$ there would exist $f(X) \in \mathcal A^{0}[X]$ with at least one coefficient of $P$-adic valuation zero such that $f(\mu(t))=0.$
Since $sgn (\mu(t))=\mu(t)$ we would get that $\mu(t)$ is algebraic over $\overline{\mathbb F_q}(t)$ which is a contradiction.
\end{proof}
\begin{lemma}\label{lemma7}${}$\par
\noindent Let $\omega(t) =\sum_{j\geq 0}\lambda_{(\theta+\varepsilon_P)^{j+1}}(t+\varepsilon_P)^j \in \mathcal B^0.$ Then we have the following equality in $\mathcal B^0 :$
$$\omega_P(t) \mu(t) = \omega(t).$$
\end{lemma} 
\begin{proof} First observe that :
$$\tau (\omega(t))= (t-\theta) \omega(t),$$
and :
$$\mu(t)=sgn(\omega(t)).$$
Thus by Lemma \ref{lemma1} :
$$\omega(t) =\mu(t) (1+Pf),$$
for some $f\in \mathcal B^0.$
Therefore :
$$\frac{\omega(t)}{\mu(t)} \in \mathcal (B^0)^\times.$$
We have :
$$\tau (\frac{\omega_P(t)\mu(t)}{ \omega(t)})= \frac{\omega_P(t)\mu(t)}{ \omega(t)}.$$
Thus :
$$\frac{\omega_P(t)\mu(t)}{ \omega(t)}\in \mathbb F_q[[t+\varepsilon_P]].$$
But :
$$v_P(\frac{\omega_P(t)\mu(t)}{ \omega(t)}-1)>0.$$
The Lemma follows.
\end{proof}
\noindent Observe that by Lemma \ref{Prellemma1}, $\omega(t)$ is the usual Anderson-Thakur function viewed as an element in $\mathcal B.$ 

\noindent Let $C^{\dag}: A\rightarrow \mathcal B^{0} \{ \tau\}$ be the morphism of $\mathbb F_q$-algebras given by :
$$C^{\dag} _{\theta} =(t-\zeta_P) \tau +\theta.$$
Then we have the following equality in $\mathcal B^{0}\{ \tau \}:$
$$ C_{\theta}\mu(t)= \mu(t) C^{\dag }_{\theta}.$$
We set :
$$log_{C^{\dag}}= \sum_{i\geq 0} \frac{sgn(b_i(t))}{\ell_i }\tau ^i \in \mathcal B\{ \{ \tau \}\}.$$
Then, we have the following equality in $\mathcal B\{ \{ \tau \}\} :$
$$  log_{C^{\dag}} C^{\dag}_{\theta} = \theta log_{C^{\dag}}                           ,$$
and :
\begin{equation}\label{logCdag}   \mu(t) log_{C^{\dag}}=  log_C \mu(t)                                .\end{equation}
Now, we extend $log_{C,P}: O_{\mathbb C_P}\rightarrow \mathbb C_P$ into a morphism of $\mathbb F_q[[t+\varepsilon_P]]$-algebras $log_{C, P}: \mathcal B^{0} \rightarrow \mathbb C_P[[t+\varepsilon_P]],$ by setting :
$$log_{C,P}(\sum_{i\geq 0} \alpha_i (t+\varepsilon_P)^i )=\sum_{i\geq 0} log_{C,P}( \alpha_i) (t+\varepsilon_P)^i,\,  \alpha_i \in O_{\mathbb C_P}.$$
Observe that by Lemma \ref{lemma7}, we have :
$$log_{C,P}(\omega_P(t) \mu(t))=0.$$

Let $\mathcal M_{\mathcal A}=\{ x\in \mathcal A, v_P(x)>0\}.$ Observe that :
$$\mathcal A^0=\{ x\in \mathcal A, v_P(x)\geq 0\}= \overline{\mathbb F_q}[t, \frac{1}{P(t)}]\oplus \mathcal M_{\mathcal A}.$$
We notice that $log_{C^{\dag}}$ converges on $\mathcal M_{\mathcal A}.$ We have :
$$C^{\dag}_P\equiv P(t) \tau^d \pmod{PA_P[t]\{ \tau\}}.$$
Thus, if $\alpha \in \overline{\mathbb F_q},$ there exists an integer $n\geq 1$ such that :
$$v_P(C^{\dag}_{P^n}(\alpha) -P(t)^n \alpha)\geq 1.$$
Therefore, as in section 3, we set :
$$log_{C^{\dag}, P}(\alpha )= \frac{1}{P^n-P(t)^n}log_{C^{\dag}}(  C^{\dag}_{P^n}(\alpha) -P(t)^n \alpha          )\in \mathcal A .$$
We obtain a map: $log_{C^{\dag}, P}: \mathcal A^0\rightarrow \mathcal A.$ 
\begin{lemma}\label{lemma8}
We have :
$$log_{C^{\dag},P}(\omega_P(t)) =0.$$
In fact, we have :
$$\forall x\in \mathcal A^0, \mu(t) log_{C^{\dag}, P}(x)=log_{C,P}(x\mu(t)).$$
\end{lemma}
\begin{proof} Observe that :
$$C^{\dag}_{\theta} (\omega_P(t))= t\omega_P(t).$$
By the definition of $log_{C^{\dag}, P},$ we have :
$$log_{C^{\dag},P}(\omega_P(t))= \frac{1}{\theta-t}log_{C^{\dag}, P}( C^{\dag}_{\theta}(\omega_P(t))-t\omega_P(t))=0.$$
The second assertion can be proved  the same way and is left to the reader.
\end{proof}

\noindent Now, we set :
$$\mathcal L_P(t)= \prod_{j=0}^{d-1} \tau ^j (\mathcal L (\mathbb F_{q^d}; \frac{1}{t-\zeta_P}; \theta-\zeta_P))\in \mathbb F_{q^d}[t, \frac{1}{P(t)}][[P]].$$
Observe that $\mathcal L_p(t)$ is convergent for $t\in \mathbb C_P\setminus \{ \zeta_P, \ldots , \zeta_p^{q^d-1}\}.$

\begin{theorem}\label{theorem4}${}$\par
 \noindent Set $\pi_P=\prod_{i=0}^{d-1}\prod_{j\geq 1}(1- (\theta-\zeta_P)^{q^{dj+i}-q^i})^{-1}.$ We have the following equality in $\mathbb F_{q^d}[t, \frac{1}{P(t)}]((P)):$
 $$\frac{b_d(t)\mathcal L_P(t)}{\pi_P}(  \frac{1}{P}log_{C^{\dag},P}  ( \sum_{k=1}^{d} k(P,k)b_k(t) \omega_P(t)) )                                      =  P(t)L_P(t),$$
 where $C_P=\sum_{k=0}^d (P,k) \tau^k, $ $(P,k)\in A, k=0, \ldots, d.$
\end{theorem}
\begin{proof}
\noindent Set :
$$Log_{P,z}=\sum_{a\in A_{+}\setminus PA_{+}} \frac{z^{deg a}}{a} C_a \in A_P\{ \tau \}[[z]],$$
$$log_{C,z}= \sum_{k\geq 0} \frac{z^k}{\ell_k} \tau^k \in K \{ \tau \}[[z]].$$
Here $\tau z= z \tau .$ By Anderson' s log algebraicity Theorem (see \cite{AND}):
$$log_{C,z}=\sum_{a\in A_{+}} \frac{z^{deg a}}{a} C_a.$$
Therefore :
$$Log_{P,z}= log_{C,z}-\frac{z^d}{P}  log_{C,z} C_P.$$
We get in $\mathcal B[[z]]:$
\begin{equation}\label{*} Log_{P,z}({ \omega_P(t) }{\mu(t)})=(\sum_{a\in A_{+}\setminus PA_{+}}\frac{a(t)z^{deg a }}{a}) ({\omega_P(t) }{\mu(t)}).\end{equation}
We have :
$$ \frac{d}{dz}Log_{P,z}= \frac{d}{dz} log_{C,z}-d\frac{z^{d-1}}{P}log_{C,z}C_P - \frac{z^d}{P}(\frac{d}{dz} log_{C,z}) C_P.$$
Taking $z=1$ we get :
$$  \frac{d}{dz}Log_{P,z}\mid_{z=1}= \frac{d}{dz} log_{C,z}\mid_{z=1}-dlog_C -\frac{1}{P}\frac{d}{dz} log_{C,z}\mid_{z=1} C_P \in K\{ \{ \tau \}\}.$$
But, using the fact that :
$$log_{C,z} C_{z,P} = Plog_{C,z},$$
where if $C_P=\sum_{k=0}^d (P,k) \tau ^k,$ we have :
$$C_{z,P}= \sum_{k=0}^d z^k (P,k)\tau^k .$$
We get :
$$ \frac{d}{dz}Log_{P,z}\mid_{z=1}= \frac{1}{P} log_C ((\sum_{k=1}^{d} k(P,k)\tau ^k) - d log_C  \in K\{ \{ \tau \}\}.$$
Therefore, evaluating at $\mu(t)\omega_P(t)$ and using \eqref{*} and \eqref{logCdag}, we get :
$$          \frac{1}{P}log_{C^{\dag}, P}  ( \sum_{k=1}^{d} k(P,k)b_k(t) \omega_P(t))-dlog_{C^{\dag},P}(\omega_P(t))=                                               L_P(t)\omega_P(t) .$$
But  by Proposition \ref{Preltheorem1} :
$$\frac{\mathcal L_P(t) \omega_P(t)}{\pi_P}= \frac{P(t)}{b_d(t)}.$$
We get the result by Lemma \ref{lemma8}.
\end{proof}


\subsection{The case $s\geq 2$}
Let $t_1,\ldots, t_s,z$ be $s+1$ indeterminates over $\mathbb C_P$ and let $\mathcal B_s$ be the set of elements in $\mathbb C_P[[t_1+\varepsilon_P, \ldots, t_s+\varepsilon_P]]$ with bounded coefficients equipped with the usual Gauss $P$-adic valuation. Let $\tau : \mathcal B_s\rightarrow \mathcal B_s$ be the continuous morphism of $\mathbb F_q[[t_1+\varepsilon_P, \ldots, t_s+\varepsilon_P]]$-algebras given by :
$$\forall x\in \mathbb C_p, \tau (x)=x^q.$$
Let $sgn :\mathcal B_s^0:=\{ x\in \mathcal B_s, v_p(x)\geq 0\} \rightarrow \overline{\mathbb F_q}[[t_1+\varepsilon_P, \ldots, t_s+\varepsilon_P]]$ be the continuous morphism of  $\overline{\mathbb F_q}[[t_1+\varepsilon_P, \ldots, t_s+\varepsilon_P]]$-algebras  that extends $sgn :O_{\mathbb C_P}\rightarrow \overline{\mathbb F_q}.$ Let also $\mathcal B_s^{00}={\rm ker} sgn.$
 
Let $\mathcal A_s\subset \mathcal B_s$ be the $P$-adic completion of $\mathbb C_P[t_1, \ldots, t_s, \frac{1}{P(t_1)\cdots P(t_s)}].$ We also set $\mathcal A_s^0=\{ x\in \mathcal A_s, v_P(x)\geq 0\}$ and $\mathcal A_s^{00}=\{ x\in \mathcal A_s, v_P(x)>0\}.$
Let $\phi: A\rightarrow \mathcal A_s\{ \tau\}$ be the morphism of $\mathbb F_q$-algebras given by :
$$\phi_{\theta}= (t_1-\zeta_P)\cdots (t_s-\zeta_P)\tau +\theta.$$
If $s=1$ then $\phi =C^{\dag}.$ Let :
$$log_{\phi}=\sum_{i\geq 0}\frac{sgn(b_i(t_1)\cdots b_i(t_s))}{\ell_i} \tau ^i\in \mathcal A_s\{ \{ \tau \}\}.$$
We have in $\mathcal B\{ \{ \tau \}\} :$
$$\mu(t_1) \cdots \mu(t_s) \phi_{\theta}= C_{\theta} \mu(t_1) \cdots \mu(t_s),$$
$$\mu(t_1) \cdots \mu(t_s)log_{\phi} = log_C\mu(t_1) \cdots \mu(t_s),$$
$$log_{\phi} \phi_{\theta}= \theta log_{\phi}.$$
As in the previous section let $log_{C,P}: \mathcal B_s^0\rightarrow \mathbb C_P[[t_1+\varepsilon_P, \ldots, t_s+\varepsilon_P]]$ be the continuous morphism of $\mathbb F_q[[t_1+\varepsilon_P, \ldots, t_s+\varepsilon_P]]$-algebras induced by $log_{C,P}: O_{\mathbb C_P}\rightarrow \mathbb C_P.$
Now, observe that $log_{\phi}$ converges on $\mathcal A_s^{00}$ and :
$$\phi_P\equiv P(t_1)\cdots P(t_s)\tau ^d \pmod {P\mathcal A^0\{ \tau\}}.$$
Thus, again as in the previous section, we can extend $log_{\phi}$ to a continuous morphism of $\mathbb F_q[t_1, \ldots, t_s, \frac{1}{P(t_1)\cdots P(t_s)}]$-algebras $log_{\phi, P}: \mathcal A_s^0 \rightarrow \mathcal A.$ We have :
$$\forall x\in \mathcal A^0, \mu(t_1)\cdots \mu(t_s) log_{\phi, P}(x) =log_{C,P}(x\mu(t_1)\cdots \mu(t_s)).$$
We will follow the proof of Theorem \ref{theorem2}. We work in $\mathcal B_s[[z]]$ where $\tau z =z\tau.$ We have in $\mathcal B_s[[z]]:$
$$\sum_{k\geq 0}\sum_{a\in A_{+,k}}\frac{a(t_1)\cdots a(t_s)}{a} z^k = log_{s,z} u(t_1, \ldots, t_s;z),$$
where :
$$log_{s,z}=\sum_{i\geq 0}\frac{b_i(t_1)\cdots b_i(t_s)z^i}{\ell_i}\tau ^i.$$
Thus:
$$\sum_{k\geq 0}\sum_{\substack{a\in A_{+,k} \\ a\not \equiv 0\pmod{P}}}\frac{a(t_1)\cdots a(t_s)}{a} z^k=(1-\frac{P(t_1)\cdots P(t_s)z^d}{P})  log_{s,z} u(t_1, \ldots, t_s;z).$$
We multiply the above equality by $\omega(t_1)\cdots \omega(t_s) \in \mathcal B^0,$ we get :
\begin{align*}
&\omega(t_1)\cdots \omega(t_s)\sum_{k\geq 0}\sum_{a\in A_{+,k}, a\not \equiv 0\pmod{P}}\frac{a(t_1)\cdots a(t_s)}{a} z^k\\
&=  (1-\frac{P(t_1)\cdots P(t_s)z^d}{P})  log_{C,z}\omega(t_1)\cdots \omega(t_s)  u(t_1, \ldots, t_s;z), 
\end{align*}
where we recall that $log_{C,z}= \sum_{i\geq 0} \frac{z^i}{\ell_i}\tau ^i.$ Observe that :
$$\omega(t_1)\cdots \omega(t_s)  u(t_1, \ldots, t_s;z)\in \mu(t_1)\cdots \mu(t_s) \mathcal A_s^0.$$
Now, we divide by $\mu(t_1)\cdots \mu(t_s),$ we have (in    $\mathcal A[[z]] $): 
\begin{align*}
&\omega_P(t_1)\cdots \omega_P(t_s)\sum_{k\geq 0}\sum_{a\in A_{+,k}, a\not \equiv 0\pmod{P}}\frac{a(t_1)\cdots a(t_s)}{a} z^k \\ &= (1-\frac{P(t_1)\cdots P(t_s)z^d}{P})  log_{\phi,z}\omega_P(t_1)\cdots \omega_P(t_s)  u(t_1, \ldots, t_s;z),\end{align*}
where :
$$log_{\phi,z}=\sum_{i\geq 0}\frac{sgn(b_i(t_1)\cdots b_i(t_s))z^i}{\ell_i}\tau ^i.$$
Now, as in the proof of Theorem \ref{theorem2}, 
if  $s\not \equiv 1\pmod{q-1},$ we have :
\begin{align*}
&\omega_P(t_1)\cdots \omega_P(t_s)L_P(t_1, \ldots, t_s)\\
&= (1-\frac{P(t_1)\cdots P(t_s)}{P})log_{\phi,P}(\omega_P(t_1)\cdots \omega_P(t_s)u(t_1, \ldots, t_s;1)),
\end{align*}
and if  $s \equiv 1\pmod{q-1},$ we get :
\begin{align*}
&\omega_P(t_1)\cdots \omega_P(t_s)L_P(t_1, \ldots, t_s)\\&= (1-\frac{P(t_1)\cdots P(t_s)}{P})log_{\phi,P}(\omega_P(t_1)\cdots \omega_P(t_s)\frac{d}{dz}u(t_1, \ldots, t_s;z)\mid_{z=1}).
\end{align*}
Therefore, for $s\not \equiv 1\pmod{q-1},$ if we set :
\begin{eqnarray*}
U(t_1, \ldots, t_s)&=&\omega_P(t_1)\cdots \omega_P(t_s)u(t_1, \ldots, t_s;1)\\
&\in& \mathbb F_{q^d}[t_1,\ldots, t_s, \frac{1}{P(t_1)\cdots P(t_s)}][[P]],
\end{eqnarray*}
and for $s\equiv 1\pod{q-1}:$
\begin{eqnarray*}
U(t_1, \ldots, t_s)&=&\omega_P(t_1)\cdots \omega_P(t_s)\frac{d}{dz}u(t_1, \ldots, t_s;z)\mid_{z=1}\\
&\in& \mathbb F_{q^d}[t_1,\ldots, t_s, \frac{1}{P(t_1)\cdots P(t_s)}][[P]],
\end{eqnarray*}
We get for $s\geq 2:$
\begin{align*}
&P(t_1)\cdots P(t_s)L_P(t_1, \ldots, t_s)\\ &= b_d(t_1)\cdots b_d(t_s)\frac{\mathcal L_P(t_1)\cdots \mathcal L_P(t_s)}{\pi_P^s}(1-\frac{P(t_1)\cdots P(t_s)}{P})log_{\phi,P}(U(t_1, \ldots,t_s)).
\end{align*}
An interesting case is when $s\geq 2,$ $s\equiv 1\pmod{q^d-1}.$ Let's set :
\begin{eqnarray*}
B_P(t_1, \ldots, t_s)&=&\prod_{j=0}^{d-1} \tau ^j (B(\mathbb F_{q^d}; \frac{1}{t_1-\zeta_P}, \ldots, \frac{1}{t_s-\zeta_P}; \theta-\zeta_P))\\
&\in& \mathbb F_{q^d}[t_1, \ldots, t_s, \frac{1}{P(t_1)\cdots P(t_s)}][\theta].
\end{eqnarray*}
We get in this case :
\begin{align*}
& B_P(t_1,\ldots, t_s)L_P(t_1, \ldots, t_s)\\
&= \frac{\mathcal L_P(t_1, \ldots, t_s)}{\pi_P}(1-\frac{P(t_1)\cdots P(t_s)}{P})log_{\phi,P}(U(t_1, \ldots,t_s)),
\end{align*}
where :
$$\mathcal L_P(t_1, \ldots, t_s)=\prod_{j=0}^{d-1}\tau ^j (\mathcal L(\mathbb F_{q^d}; \frac{1}{t_1-\zeta_P}, \ldots, \frac{1}{t_s-\zeta_P}; \theta-\zeta_P)).$$


\section{The case of the place $\infty$}
We are now applying the results in section 6 to get some new functional identities for $L$-series introduced in \cite{PEL}.
\subsection{The case $s=1$}${}$\par
Let $\mathcal A_{\infty}$ be the completion, for the Gauss valuation $v_{\infty},$ of $\mathbb C_{\infty}[t, \frac{1}{t}].$ Let $\mathcal B_{\infty}\subset \mathbb C_{\infty}[[t+1]]$ be the set of power series with bouded coefficients. Then :
$$\mathcal A_{\infty}\subset \mathcal B_{\infty}.$$
Set $\mathcal B_{\infty}^0=\{ x\in \mathcal B_{\infty}, v_{\infty} (x)\geq 0\}=O_{\mathbb C_{\infty}}[[t+1]].$ 
Let $sgn=sgn_{\infty}: \mathcal B_{\infty}^0\rightarrow\overline{\mathbb F_q}[[t+1]]$ be the continuous morphism of $\overline{\mathbb F_q}[[t+1]]$-algebras such that:
$$\forall x\in \mathbb C_{\infty}, v_{\infty} (x-sgn(x))>0.$$
We set :
$$\mathcal B_{\infty}^{00}={\rm Ker} sgn .$$
Set $\mathcal A_{\infty}^0=\mathcal A_{\infty}\cap \mathcal B_{\infty}^0$ and $\mathcal A_{\infty}^{00}= \mathcal A_{\infty}\cap \mathcal B_{\infty}^{00}= \{ x\in \mathcal A_{\infty}, v_{\infty}(x)>0\}.$ Observe that :
$$\mathcal A_{\infty}^0=\overline{\mathbb F_q}[t, \frac{1}{t}]\oplus \mathcal A_{\infty}^{00}.$$
Recall that :
$$\omega(t)=\lambda_{\theta}\prod_{j\geq 0}(1-\frac{t}{\theta ^{q^j}})^{-1}\in \mathcal A_{\infty}^\times.$$
Let's set :
$$\omega_{\infty}(t) = \frac{\omega(\frac{1}{t})}{\lambda_{\theta}}\in( \mathcal A_{\infty}^0)^\times.$$
We first introduce a new kind of $L$-series.  We set :
$$\mathcal L_{\infty} (t)=\sum_{k\geq 1}\sum_{a\in A_{+,k}, a(0)\not =0}k\frac{a(t)}{a(\frac{1}{\theta})}.$$
then by Theorem \ref{Preltheorem2}, we have :
$$\mathcal L_{\infty}(t) \in \mathcal A_{\infty}^{0}.$$
And in fact $\mathcal L_{\infty}(t)$ defines an entire function on $\mathbb C_{\infty}.$ Also recall that (see \cite{PEL}) :
$$L(t)=\sum_{k\geq 0} \sum_{a\in A_{+,k}}\frac{a(t)}{a} \in (\mathcal A_{\infty}^0)^\times.$$
Let $\tau : \mathcal B_{\infty}\rightarrow \mathcal B_{\infty}$ be the continuous morphism of $\mathbb F_q[[t+1]]$-algebras given by :
$$\forall x\in \mathbb C_{\infty}, \tau (x)=x^q.$$
Now, let $C^{\dag} :A\rightarrow \mathcal A_{\infty}\{ \tau\}$ be the morphism of $\mathbb F_q$-algebras given by :
$$C^{\dag}_{\theta} = t\tau +\frac{1}{\theta}.$$
We set :
$$log_{C^{\dag}}=1+\sum_{j\geq 1}\frac {t^j }{\prod_{k=1}^{j}(\frac{1}{\theta}-\frac{1}{\theta^{q^k}})}\tau ^j \in \mathcal A_{\infty}\{ \{ \tau \}\}.$$
We have in $\mathcal A_{\infty}\{ \{ \tau \}\} :$
$$log_{C^{\dag}}C^{\dag}_{\theta}= \frac{1}{\theta} log_{C^{\dag}}.$$
Observe that $log_{C^{\dag}}$ converges on $\mathcal A_{\infty}^{00}.$ Furthermore, if $\alpha \in \mathbb F_{q^n}, n\geq 1,$ then :
$$v_{\infty} (C^{\dag}_{\theta^n}(\alpha) -t^n \alpha)\geq 1.$$
Therefore we set :
$$log_{C^{\dag}, \infty}(\alpha )=\frac{1}{\theta^{-n}-t^n}log_{C^{\dag}}(C^{\dag}_{\theta^n}(\alpha) -t^n \alpha)\in \mathcal A_{\infty}.$$
Thus, as in section 3, we obtain a continuous morpshism of $\mathbb F_q[t, \frac{1}{t}]$-modules :
$$log_{C^{\dag}, \infty}: \mathcal A_{\infty}^0\rightarrow \mathcal A_{\infty}.$$
Observe that :
$$log_{C^{\dag}, \infty}(\omega_{\infty}(t))=0.$$
The reader can recognize that we are in the situation of section 6 for the local field $\mathbb F_q((\frac{1}{\theta})).$ Therefore if we apply Theorem \ref{theorem4} (change $\theta$ in $\frac{1}{\theta}$ in the formula), we get :
\begin{corollary}\label{corollary1} We have the following equality in $\mathbb F_q[t, \frac{1}{t}]((\frac{1}{\theta})):$
$$   L(\frac{1}{t}) \, log_{C^{\dag}, \infty}((t-\frac{1}{\theta})\frac{\omega(\frac{1}{t})}{\lambda_{\theta}}) =\frac{t\widetilde{\pi} }{\theta \lambda_{\theta}(\theta  t-1) }\, \mathcal L_{\infty}(t)                                      .$$

\end{corollary}

\subsection{The case $s\geq 2$}${}$\par
Let $s\geq 2.$ Let $\mathcal A_{s, \infty}$ be the completion for the Gauss $\infty$-adic valuation of $\mathbb C_{\infty}[t_1, \ldots, t_s, \frac{1}{t_1\cdots t_s}].$ We set :
$$\mathcal A_{s,\infty}^0=\{ x\in \mathcal A_{s, \infty}, v_{\infty}(x)\geq 0\},$$
$$\mathcal A_{s, \infty}^{00}=\{ x \in \mathcal A_{s, \infty} , v_{\infty}(x)>0\}.$$
Then :
$$\mathcal A_{s, \infty}^0 = \overline{\mathbb F_q}[t_1, \ldots , t_s, \frac{1}{t_1\cdots t_s}]\oplus \mathcal A_{s, \infty}^{00}.$$
Let $\tau :\mathcal A_{s, \infty}\rightarrow \mathcal A_{s, \infty}$ be the continuous morphism of $
{\mathbb F_q}[t_1, \ldots , t_s, \frac{1}{t_1\cdots t_s}]$-algebras such that :
$$\forall x\in \mathbb C_{\infty}, \tau (x)=x^q.$$
Let $\phi : A\rightarrow \mathcal A_{s, \infty}\{ \tau\}$ be the morphism of $\mathbb F_q$-algebras such that :
$$\phi_{\theta}= t_1\cdots t_s \tau +\frac{1}{\theta}.$$
Set :
$$log_{\phi}=1+\sum_{j\geq 1}\frac {t_1^j\cdots t_s^j }{\prod_{k=1}^{j}(\frac{1}{\theta}-\frac{1}{\theta^{q^k}})}\tau ^j \in \mathcal A_{s,\infty}\{ \{ \tau \}\}.$$
Then $log_{\phi}$ induces a continuous morphism of ${\mathbb F_q}[t_1, \ldots , t_s, \frac{1}{t_1\cdots t_s}]$-modules :
$$log_{\phi, \infty}: \mathcal A_{s, \infty}^0\rightarrow \mathcal A_{s, \infty}.$$
Let $v(t_1, \ldots, t_s; z)\in \mathbb F_q[t_1, \ldots, t_s,z, \frac{1}{\theta}]$ be the polynomial obtained by formally replacing $\theta$ by $\frac{1}{\theta}$ in the Anderson-Stark unit $u(t_1, \ldots, t_s; z)\in \mathbb F_q[t_1, \ldots , t_s, z, \theta].$
Now, we set :
$$F_{\infty} (t_1, \ldots, t_s; z)=( \prod_{j=1}^s \frac{\omega(\frac{1}{t_j})}{\lambda_{\theta}})v(t_1, \ldots, t_s; z) \in \mathcal A_{s, \infty}^0[z].$$
We further, set if $s\not \equiv 1\pmod{q-1}:$
$$U_{\infty}(t_1, \ldots, t_s)=F_{\infty} (t_1, \ldots, t_s; z)\mid_{z=1},$$
and if $s\equiv 1\pmod{q-1}:$
$$U_{\infty}(t_1, \ldots, t_s)=\frac{d}{dz}F_{\infty} (t_1, \ldots, t_s; z)\mid_{z=1}.$$
Now, we introduce a new kind of $L$-series. If $s\equiv 1\mod{q-1}:$
$$\mathcal L_{\infty}(t_1, \ldots, t_s)=\sum_{k\geq 1}\sum_{ a\in A_{+,k}, a(0)\not =0} k\frac{a(t_1) \cdots a(t_s)}{a(\frac{1}{\theta})} \in \mathcal A_{s, \infty}^0,$$
and if $s\not \equiv 1\pmod{q-1}:$
$$\mathcal L_{\infty}(t_1, \ldots, t_s)=\sum_{k\geq 0}\sum_{ a\in A_{+,k}, a(0)\not =0} \frac{a(t_1) \cdots a(t_s)}{a(\frac{1}{\theta})} \in \mathcal A_{s, \infty}^0.$$
Again, by Theorem \ref{Preltheorem1}, $\mathcal L_{\infty}(t_1, \ldots, t_s)$ defines an entire function on $\mathbb C_{\infty}^s.$ The reader recognize that we are again in the situation of section 6 for the local field $\mathbb F_q((\frac{1}{\theta})).$ Therefore we get :
\begin{corollary}\label{corollary2} We have the following equality in $\mathbb F_q[t_1, \ldots, t_s, \frac{1}{t_1\cdots t_s}]((\frac{1}{\theta})):$
\begin{align*}
&(t_1-\frac{1}{\theta})\cdots (t_s-\frac{1}{\theta})\frac{\lambda_{\theta}^s \theta^sL(\frac{1}{t_1} )\cdots L(\frac{1}{t_s})}{\widetilde{\pi}^s}
(1- \theta t_1\cdots t_s) log_{\phi, \infty} (U_{\infty} (t_1, \ldots, t_s))\\ &= t_1\cdots t_s \mathcal L_{\infty} (t_1, \ldots, t_s).
\end{align*}
\end{corollary}

\noindent Now let $s\geq 2,$ $s\equiv 1\pmod{q-1}.$ Set :
$$B_{\infty}(t_1, \ldots, t_s)= B(\mathbb F_q; \frac{1}{t_1}, \ldots, \frac{1}{t_s}; \frac{1}{\theta})\in \mathbb F_q[\frac{1}{t_1}, \ldots, \frac{1}{t_s}, \frac{1}{\theta}].$$
We get by section 6 :
\begin{corollary}\label{corollary3} We have the following equality in $\mathbb F_q[t_1, \ldots, t_s, \frac{1}{t_1\cdots t_s}]((\frac{1}{\theta})):$
\begin{align*}
&\frac{\lambda_{\theta}\theta L(\frac{1}{t_1}, \ldots, \frac{1}{t_s})}{\widetilde{\pi}}    (1- \theta t_1\cdots t_s) log_{\phi, \infty} (U_{\infty} (t_1, \ldots, t_s))\\
&= B_{\infty}(t_1, \ldots, t_s) \mathcal L_{\infty} (t_1, \ldots, t_s).
\end{align*}
\end{corollary}


\section{Application to class modules}
\subsection{Background on the  $s$-variable  class module}${}$\par
Let $s\geq 2$ be an integer. Let $t_1, \ldots, t_s,z$ be $s+1$ variables over $\mathbb C_{\infty}$ and let $\tau :\mathbb C_{\infty}[[t_1, \ldots, t_s,z]]\rightarrow \mathbb C_{\infty}[[t_1, \ldots, t_s,z]]$ be the morphism of $\mathbb F_q[[t_1, \ldots, t_s,z]]$-algebras such that :
$$\forall x\in \mathbb C_{\infty}, \tau (x)=x^q.$$
Let $\mathbb T_s(K_{\infty})$ be the Tate algebra in the variables $t_1, \ldots, t_s$ with coefficients in $K_{\infty},$ i.e. $\mathbb T_s(K_{\infty})\subset K_{\infty}[[t_1, \ldots, t_s]]$ is the completion of $K_{\infty}[t_1, \ldots,t_s]$ for the $\infty$-adic Gauss valuation. Let $\phi: A\rightarrow \mathbb  T_s(K_{\infty})\{ \tau \}$ be the morphism of $\mathbb F_q$-algebras given by :
$$\phi_{\theta}=(t_1-\theta)\cdots (t_s-\theta)\tau +\theta.$$
The exponential function associated to $\phi$ is :
$$\exp_{\phi}=\sum_{j\geq 0} \frac{b_j(t_1)\cdots b_j(t_s)}{D_j} \tau ^j\in \mathbb T_s(K_{\infty})\{ \{ \tau \}\}.$$
It satisfies the functional equation in $\mathbb T_s(K_{\infty})\{ \{ \tau \}\}:$
$$\exp_{\phi} \theta =\phi_{\theta} \exp_{\phi}.$$
Furthermore $\exp_{\phi}$ converges on $\mathbb T_s(K_{\infty}).$ We set :
$$U_s=\{ f\in \mathbb T_s(K_{\infty}), \exp_{\phi}(f)\in A[\underline{t}_s]\},$$
$$H_s= \frac{\mathbb T_s(K_{\infty})}{A[\underline{t}_s]+\exp_{\phi}(\mathbb T_s(K_{\infty}))}.$$
It is proved in \cite{APTR}, Corollary 5.2 and Corollary 5.5,   that $U_s$ is a free $A[\underline{t}_s]$-module of rank one and $H_s$ is a finitely generated  and torsion $A[\underline{t}_s]$-module. Furthermore, if $s\equiv 1\pmod{q-1}$ (see \cite{APTR}, section 6):
$$U_s= \frac{\widetilde{\pi}}{\omega(t_1)\cdots \omega(t_s)}A[\underline{t}_s].$$
Observe that in this case, any element of $U_s$ is an entire function in $\mathbb E_s(K_{\infty})$ (see section 1). In fact, we have :
\begin{proposition} \label{ATRproposition1}
$$U_s\subset \mathbb E_s(K_{\infty}).$$
\end{proposition}
\begin{proof} Recall that by \cite{APTR}, Theorem 5.11,  $L(t_1, \ldots, t_s)\in U_s$ and  $\frac{U_s}{L(t_1, \ldots, t_s) A[\underline{t}_s]}$ is a finitely generated and torsion $A[\underline{t}_s]$-module. Note also that, since $L(t_1, \ldots, t_s)\in \mathbb T_s(K_{\infty})^\times,$   $\frac{U_s}{L(t_1, \ldots, t_s) A[\underline{t}_s]}$ is a free and finitely generated $\mathbb F_q[\underline{t}_s]$-module (see also \cite{APTR}, Remark 5.17). Therefore there exists $F_s\in A[\underline{t}_s]\setminus\{ 0\}$ which is monic as a polynomial in $\theta$ such that :
$${\rm Fitt}_{A[\underline{t}_s]}\frac{U_s}{L(t_1, \ldots, t_s)A[\underline{t}_s]} = F_sA[\underline{t}_s],$$
where ${\rm Fitt}_{A[\underline{t}_s]}M$ denotes the Fitting ideal of an $A[\underline{t}_s]$-module $M.$
 Since we have  $L(t_1, \ldots, t_s)\in \mathbb E_s(K_{\infty}),$ the Proposition follows from Proposition \ref{ATproposition1} and Lemma \ref{ATlemma3}.
\end{proof}
\noindent Let :
$$\exp_{\phi,z}=\sum_{j\geq 0} \frac{z^jb_j(t_1)\cdots b_j(t_s)}{D_j} \tau ^j\in \mathbb T_s(K_{\infty})[z]\{ \{ \tau \}\}.$$
We have :
\begin{equation}\label{exp}\exp_{\phi,z} \theta = (z(t_1-\theta)\cdots (t_s-\theta) \tau +\theta )\exp_{\phi,z}.\end{equation}
By \cite{APTR}, Corollary 5.12  :
$$\exp_{\phi,z} (L(t_1, \ldots, t_s; z))= u(t_1, \ldots, t_s; z),$$
where we recall that:
$$L(t_1, \ldots, t_s; z) =\sum_{k\geq 0}\sum_{A\in A_{+,k}}\frac{a(t_1)\cdots a(t_s)}{a} z^k,$$
and $u(t_1, \ldots, t_s; z) \in A[t_1, \ldots, t_s, z].$
If $s\not \equiv 1\pmod{q-1},$ we set :
$$u_s=u(t_1, \ldots, t_s; 1)\in A[\underline{t}_s],$$
and if $s\equiv 1\pmod{q-1}:$
$$u_s= \frac{d}{dz} u(t_1, \ldots, t_s; z)\mid_{z=1}\in A[\underline{t}_s].$$
\begin{lemma}\label{ATRlemma1}
Set $R=\mathbb F_q(t_1, \ldots,t_s)[\theta],$ we extend $\phi$ to a morphism of $\mathbb F_q(t_1, \ldots, t_s)$-algebras $\phi : R\rightarrow R\{\tau \}.$  Then $u_s$ is not a torsion point for $\phi.$
\end{lemma}
\begin{proof} Assume the contrary and let's suppose that there exists $c\in R\setminus\{0\}$ such that: $\phi_c(u_s)=0.$ Let $\zeta_1, \ldots, \zeta_s\in \overline{\mathbb F_q}$ such that the evaluation of $c$ at $t_i=\zeta_i, i=1, \ldots,s,$ is well-defined and non-zero, and let $d\in \mathbb F_q(\zeta_1, \ldots, \zeta_s)[\theta]\setminus \{ 0\}$ be this evaluation. Set :
$$a=N_{\mathbb F_q(\zeta_1, \ldots, \zeta_s)(\theta)/K}(d)\in A\setminus \{0\},$$
where: $N_{\mathbb F_q(\zeta_1, \ldots, \zeta_s)(\theta)/K}: \mathbb F_q(\zeta_1, \ldots, \zeta_s)(\theta)\rightarrow K$ is the norm map.
Let $\chi$ be the Dirichlet character of conductor $f$ associated to $\zeta_1, \ldots, \zeta_s$ (see section 1). Let $\sigma =\tau \otimes 1$ on $A[\lambda_f]\otimes_{\mathbb F_q}\mathbb F_q(\zeta_1, \ldots , \zeta_s).$ Now, let $\widetilde{C}:\mathbb F_q(\zeta_1, \ldots, \zeta_s)[\theta]\rightarrow \mathbb F_q(\zeta_1, \ldots, \zeta_s)[\theta]\{Ê\sigma\}$ be the morphism of $\mathbb F_q(\zeta_1, \ldots, \zeta_s)$-algebras such that :
$$\widetilde{C}_{\theta}= \sigma +\theta.$$
Then :
$$\widetilde{C}_{a}(g(\chi)u_s\mid_{t_i=\zeta_i, i=1, \ldots,s})=0.$$
We get  (see section 4), if $s\not \equiv 1\pmod{q-1}:$
$$C_a(u_{[\chi]}(1))=0,$$
and if $s\equiv 1\pmod{q-1} :$
$$C_a(u^{(1)}_{[\chi]})=0,$$
which is a contradiction by the results in section 4.
\end{proof}

\subsection{Pseudo-cyclicity}${}$\par
Now let $s\geq 2,$ $s\equiv 1\pmod{q-1}.$ 
\noindent We set :
$$\exp_{\phi}^{(1)}=\sum_{j\geq 1}j\frac{b_j(t_1)\cdots b_j(t_s)}{D_j} \tau ^j\in \mathbb T_s(K_{\infty})\{ \{ \tau \}\}.$$
Observe that $\exp_{\phi}^{(1)}$ converges on $\mathbb T_s(K_{\infty}).$ Furthermore, if we take the derivative of formula \eqref{exp} with respect to $z$ and evaluate at $z=1,$  we get  :
$$\exp_{\phi}^{(1)} \theta =(t_1-\theta) \cdots (t_s-\theta)\tau  \exp_{\phi}  +  \phi_{\theta} \exp_{\phi}^{(1)}.$$
\begin{proposition}\label{ATRproposition2} The map induced by $\exp_{\phi}^{(1)} :$
$$\exp_{\phi}^{(1)}: U_s \rightarrow \frac{\mathbb T_s(K_{\infty})}{\exp_{\phi}(\mathbb T_s(K_{\infty}))},$$
is an injective morphism of $A$-modules.
\end{proposition}
\begin{proof} 
We have : 
$$\exp_{\phi}^{(1)}(L(t_1, \ldots, t_s))\equiv u_s\pmod{\exp_{\phi}(\mathbb T_s(K_{\infty}))}.$$ 
Furthermore, if $a\in A,$ we have :
$$\exp_{\phi}^{(1)}(L(t_1, \ldots, t_s)a)\equiv \phi_{a}(u_s)\pmod{\exp_{\phi}(\mathbb T_s(K_{\infty}))}.$$ 
\noindent Since $s\geq 2,$ $s\equiv 1\pmod{q-1},$ we have by  \cite{APTR}, proof of Proposition 7.2 :
$$A[\underline{t}_s]\cap \exp_{\phi}(\mathbb T_s(K_{\infty}))=\{0\}.$$
Therefore, for $a\in A[\underline{t}_s] ,$ we get that $\exp_{\phi}^{(1)}(L(t_1, \ldots, t_s)a)\equiv0\mod{\exp_{\phi}(\mathbb T_s(K_{\infty}))}$ if and only if $\phi_{a} (u_s)=0.$ But , by Lemma \ref{ATRlemma1},  we have that  $\phi_{a} (u_s)=0$ if and only if $a=0.$ We deduce that the map induced by $\exp_{\phi}^{(1)}:$
$$L(t_1, \ldots, t_s) A[\underline{t}_s] \rightarrow \frac{\mathbb T_s(K_{\infty})}{\exp_{\phi}(\mathbb T_s(K_{\infty}))},$$
is an injective morphism of $A$-modules. Since there exists $F_s\in A[\underline{t}_s]$ which is monic as a polynomial in $\theta$ such that:
$$U_s=\frac{1}{F_s}L(t_1, \ldots, t_s) A[\underline{t}_s],$$
and since :
$$\exp_{\phi,z} (\frac{1}{F_s}L(t_1, \ldots, t_s;z))\mid_{z=1}=0,$$
we deduce easily the assertion of the Proposition.
\end{proof}
\begin{lemma}\label{ATRlemma2}${}$\par
\noindent 
 We have:
$$\exp_{\phi}^{(1)}(\frac{\widetilde{\pi}}{\omega(t_1)\cdots \omega(t_s)})\equiv 
\frac{(-1)^{\frac{s-q}{q-1}}}{\theta^{\frac{s-q}{q-1}}}\pmod{ \exp_{\phi}(\mathbb T_s(K_{\infty}))}.$$
\end{lemma}
\begin{proof}${}$\par
\noindent   By \cite{APTR}, proof of Proposition 7.2,   we have:
$$ \exp_{\phi}(\mathbb T_s(K_{\infty}))=\{ x\in \mathbb T_s(K_{\infty}), v_{\infty}(x)>\frac{s-q}{q-1}\}.$$
Now observe that :
$$v_{\infty}(\frac{\widetilde{\pi}}{\omega(t_1)\cdots \omega(t_s)}+\frac{(-1)^{\frac{s-q}{q-1}}}{\theta^{\frac{s-q}{q-1}}})> \frac{s-q}{q-1}.$$
For $i\geq 2:$
$$v_{\infty}( \frac{b_i(t_1)\cdots b_i(t_s)\tau^i(\frac{\widetilde{\pi}}{\omega(t_1)\cdots \omega(t_s)}  )}{D_i})>\frac{s-q}{q-1}.$$
Furthermore, recall that $\tau (\omega(t_i))=(t_i-\theta) \omega(t_i).$
Finally, observe that :
$$v_{\infty}(\frac{\widetilde{\pi}^q}{(\theta^q-\theta)\omega(t_1)\cdots \omega(t_s)}-\frac{(-1)^{\frac{s-q}{q-1}}}{\theta^{\frac{s-q}{q-1}}})> \frac{s-q}{q-1}.$$
\end{proof}

\begin{lemma}\label{ATRlemma3}  Let $a\in A[\underline{t}_s]\setminus\{ 0\},$ $a$ monic as a polynomial in $\theta.$  If $\phi_{a}(\frac{(-1)^{\frac{s-q}{q-1}}}{\theta^{\frac{s-q}{q-1}}})\in A[\underline{t}_s]+\exp_{\phi}(\mathbb T_s(K_{\infty}))$ then $deg_{\theta}(a)\geq \frac{s-q}{q-1}.$
\end{lemma}
\begin{proof} 
Observe that if $1\leq \ell \leq \frac{s-q}{q-1},$ we have :
$$\phi_{\theta}(\frac{1}{\theta^{\ell}})\mid_{t_i=0, i=1, \ldots, s}\equiv \frac{1}{\theta^{\ell-1}} \pmod{ \frac{1}{\theta^{\ell-2}}A}.$$
Write $a=\theta^r+\sum_{i=0}^{r-1} a_i \theta ^i, $ $ a_i \in \mathbb F_q[\underline{t}_s].$  Set $b=a\mid_{t_j=0, j= 1, \cdots s}.$ 
Let $\psi :A\rightarrow A\{ \tau\}$ be the Drinfeld module of rank one given by $\psi_{\theta}= (-1)^s\theta^s \tau +\theta.$ Observe that:
$$\exp_{\psi}(K_{\infty})=\{ x\in K_{\infty}, v_{\infty}(x)>\frac{s-q}{q-1}\}.$$
Since $\phi_{a}(\frac{1}{\theta^{\frac{s-q}{q-1}}})\in A[\underline{t}_s]+\exp_{\phi}(\mathbb T_s(K_{\infty})),$  we have :
$$\psi_b(\frac{1}{\theta^{\frac{s-q}{q-1}}})\in A+\exp_{\psi}(K_{\infty}).$$
Therefore,  we get by induction on $i:$
$$\forall i\in \{ 0, \ldots, r-1\}, a_i\mid_{t_j=0, j=1, \cdots ,s}=0.$$
Thus :
$$\phi_{\theta^r}(\frac{1}{\theta^{\frac{s-q}{q-1}}})\mid_{t_j=0, j=1, \cdots ,s} \in A+\exp_{\psi}(K_{\infty}).$$
This implies that $r\geq \frac{s-q}{q-1}.$
\end{proof}
\noindent We can answer to \cite{APTR}, question 9.9, when $s\equiv 1\pmod{q-1}:$
\begin{theorem} \label{ATRtheorem1} ${}$\par
\noindent Let $s\geq 2, s\equiv 1\pmod{q-1},$ then  $H_s\otimes_{\mathbb F_q[\underline{t}_s]}\mathbb F_q(\underline{t}_s)$ is a cyclic $R=\mathbb F_q(\underline{t}_s)[\theta]$-module.
\end{theorem}
\begin{proof}

Let $\mathcal K$ be the completion of $\mathbb F_q(\underline{t}_s)(\theta)$ for the $\infty$-adic Gauss valuation $v_{\infty}$ on $\mathbb F_q(\underline{t}_s)(\theta).$ Let $\tau :\mathcal K\rightarrow \mathcal  K$ be the continuous morphism of $\mathbb F_q(\underline{t}_s)$-algebras given by $\tau (\theta)=\theta^q.$ Recall that  $\phi : R\rightarrow R\{\tau\}$ is the morphism of $\mathbb F_q(\underline{t}_s)$-algebras such that $\phi_{\theta}=(t_1-\theta)\cdots (t_s-\theta)\tau +\theta.$ If $S\subset \mathcal K,$ we denote by $\mathbb F_q(\underline{t}_s)S$ the $\mathbb F_q(\underline{t}_s)$-vector space generated by $S.$ Since $\mathbb F_q(\underline{t}_s)\mathbb T_s(K_{\infty})$ is dense in $\mathcal K,$ we have (\cite{APTR}, Proposition 5.4):
$$\{ x\in \mathcal K, \exp_{\phi}(x) \in R\} =\mathbb F_q(\underline{t}_s)U_s,$$
and  the map $H_s\rightarrow \frac{\mathcal K}{R+\exp_{\phi}(\mathcal K)}$ (which is injective) induces  an isomorphism of $R$-modules (\cite{APTR}, section 5.4.1):
$$H_s\otimes_{\mathbb F_q[\underline{t}_s]}\mathbb F_q(\underline{t}_s)\simeq \frac{\mathcal K}{R+\exp_{\phi}(\mathcal K)}.$$
Recall that $H_s$ is a free $\mathbb F_q[\underline{t}_s]$-module of rank $\frac{s-q}{q-1}$ (\cite{APTR}, Proposition 7.2 ). Let $F_s\in A[\underline{t}_s]$ be the monic polynomial in $\theta$ of degree $\frac{s-q}{q-1}$ such that :
$${\rm Fitt} _{A[\underline{t}_s]} H_s= F_s A[\underline{t}_s].$$
By the class number formula (\cite{APTR}, Theorem 5.11):
$$U_s=\frac{1}{F_s}L(t_1, \cdots, t_s)A[\underline{t}_s].$$
Note that:
$$\exp_{\phi}(\mathcal K)=\{ x\in \mathcal K, v_{\infty}(x)>\frac{s-q}{q-1}\}.$$
Observe that $\exp_{\phi}(\mathcal K)\cap \mathbb F_q(\underline{t}_s)\mathbb T_s(K_{\infty})= \exp_{\phi}( 
\mathbb F_q(\underline{t}_s)\mathbb T_s(K_{\infty})),$ thus, by Proposition \ref{ATRproposition2}, the map $\exp_{\phi}^{(1)}$ induces an injective morphism of $R$-modules:
$$\mathbb F_q(\underline{t}_s)U_s\rightarrow \frac{\mathcal K}{\exp_{\phi}(\mathcal K)}.$$
Set:
$$V=\{ x\in \mathbb F_q(\underline{t}_s)U_s, \exp_{\phi}^{(1)}(x)\in R+\exp_{\phi}(\mathcal K)\}.$$
Then :
$$L(t_1, \cdots, t_s) R\subset V.$$
Observe that $V$ is a free module of rank one and that the monic generator $G\in R$ of the Fitting ideal ${\rm Fitt}_R\frac{\mathbb F_q(\underline{t}_s)U_s}{V}$ divides $F_s.$ Thus $G\in A[\underline{t}_s].$ This implies that:
$$\exp_{\phi}^{(1)} (\frac{G\widetilde{\pi}}{\omega(t_1)\cdots \omega(t_s)}) \in R+\exp_{\phi}(\mathcal K).$$
But $\exp_{\phi}^{(1)} (\frac{G\widetilde{\pi}}{\omega(t_1)\cdots \omega(t_s)})\in \mathbb T_s(K_{\infty}),$ and 
$$(R+\exp_{\phi}(\mathcal K))\cap \mathbb T_s(K_{\infty})=A[\underline{t}_s]+\exp_{\phi}(\mathbb T_s(K_{\infty})).$$
Thus, by Lemma \ref{ATRlemma2}:
$$\phi_{G}(\frac{(-1)^{\frac{s-q}{q-1}}}{\theta^{\frac{s-q}{q-1}}})\in A[\underline{t}_s]+\exp_{\phi}(\mathbb T_s(K_{\infty})).$$
Therefore by Lemma \ref{ATRlemma3}, we get that $deg_{\theta}G\geq \frac{s-q}{q-1},$ and since $G$ divides $F_s,$ we get $G=F_s,$ i.e.:
$$V=L(t_1, \cdots, t_s) R.$$
This implies that $\exp_{\phi}^{(1)}$ induces an injective morphism of $R$-modules:
$$\frac{\mathbb F_q(\underline{t}_s)U_s}{L(t_1, \cdots, t_s) R}\rightarrow 
\frac{\mathcal K}{R+\exp_{\phi}(\mathcal K)}.$$
But the two $\mathbb F_q(\underline{t}_s)$-vector spaces above have the same dimension, therefore the above map is an isomorphism of $R$-modules and in particular $\frac{\mathcal K}{R+\exp_{\phi}(\mathcal K)}$ is a cyclic $R$-module.
\end{proof}

\begin{corollary}\label{ATRcorollary1}
The function $\exp_{\phi}^{(1)}$ induces an injective morphism of $A$-modules :
$$\frac{U_s}{L(t_1, \ldots, t_s)A[\underline{t}_s]}\rightarrow H_s,$$
and its cokernel is a finitely generated and torsion $\mathbb F_q[\underline{t}_s]$-modules.
\end{corollary}
\begin{proof}
Let $f\in \mathbb T_s(K_{\infty})$ such that $f\in R+\exp_{\phi}(\mathcal K).$ Then there exists $a\in A[\underline{t}_s]\setminus\{ 0\}$ such that :
$$af\in A[\underline{t}_s]+\exp_{\phi}(\mathcal K).$$
But since $\exp_{\phi}(\mathcal K)\cap \mathbb T_s(K_{\infty}))=\exp_{\phi}(\mathbb T_s(K_{\infty}),$ we get:
$$af\in A[\underline{t}_s]+\exp_{\phi}(\mathbb  T_s(K_{\infty})).$$
Since $s\equiv 1\pmod{q-1},$ $H_s$ is a free $\mathbb F_q[\underline{t}_s]$-module (\cite{APTR}, Proposition 7.2), we get:
$$ f\in A[\underline{t}_s]+\exp_{\phi}(\mathbb  T_s(K_{\infty})).$$
Therefore we have an injective morphism of $A$-modules induced by the inclusion $\mathbb T_s(K_{\infty})\subset \mathcal K:$
$$H_s\rightarrow \frac{\mathcal K}{R+\exp_{\phi}(\mathcal K)}.$$
Furthermore the inclusion $U_s\subset \mathbb F_q(\underline{t}_s)U_s$ induces an injective morphism of $A$-modules :
$$\frac{U_s}{L(t_1, \ldots, t_s)A[\underline{t}_s]}\rightarrow \frac{\mathbb F_q(\underline{t}_s)U_s}{L(t_1, \ldots, t_s)R}.$$
The Corollary is a consequence of the proof of Theorem \ref{ATRtheorem1}.
\end{proof}
\subsection{The case of Taelman's class modules}${}$\par
Let $\chi$ be a Dirichlet character of conductor $f$ and type $s\geq 2.$ Recall that there exists $\zeta_1(\chi), \ldots, \zeta_s(\chi)\in \overline{\mathbb F_q}$ such that :
$$\sigma (g(\chi))= (\zeta_1(\chi)-\theta)\cdots (\zeta_s(\chi)-\theta) g(\chi),$$
where $\sigma =\tau \otimes 1$ on $A[\lambda_f]\otimes_{\mathbb F_q}\mathbb F_q(\zeta_1(\chi), \ldots, \zeta_s(\chi)).$
A property is said to be true for almost all Dirichlet characters $\chi$ of type $s$  if there exists $F(t_1,\ldots, t_s)\in A[\underline{t}_s]\setminus \{0\}$ such that the property is true for every character $\chi$ such that $F(\zeta_1(\chi), \ldots, \zeta_s(\chi))\not =0$ (see \cite{APTR}, section 9). \par
Let $\chi$ be a Dirichlet character of type $s\geq2$ and conductor $f.$ Set :
$$e_{\chi}=\frac{1}{\mid \Delta_f\mid}\sum_{\mu \in \Delta_f} \chi(\mu) \mu^{-1}\in \mathbb F_q(\zeta_1(\chi), \ldots, \zeta_s(\chi))[\Delta_f].$$ 
The $\chi$-isotypic component of Taelman's class module associated to the Carlitz module and the field $K(\lambda_f)$ is defined by (see also \cite{APTR}, section 9):
$$H_{\chi}=e_{\chi}(\frac{K(\lambda_f)\otimes_{K}K_{\infty}}{A[\lambda_f]+\exp_C(K(\lambda_f)\otimes_{K}K_{\infty})}\otimes_{\mathbb F_q}\mathbb F_q(\zeta_1(\chi), \ldots, \zeta_s(\chi))).$$
Observe that $H_{\chi}$ is a finitely generated and torsion $A[\zeta_1(\chi), \ldots, \zeta_s(\chi)]$-module.
\begin{theorem}\label{ATRtheorem2}${}$\par
\noindent Let $s$ be an integer, $s\geq 2,$ $s\equiv 1\pmod{q-1}.$ Then for almost all Dirichlet characters $\chi$ of type $s,$ $H_{\chi}$ is a cyclic  $A[\zeta_1(\chi), \ldots, \zeta_s(\chi)]$-module and generated as a  $A[\zeta_1(\chi), \ldots, \zeta_s(\chi)]$-module by :
$$\frac{g(\chi)}{\theta^{\frac{s-q}{q-1}}}.$$
\end{theorem}
\begin{proof} Let $M $ be the sub-$A$-module of $H_s$ generated by $\frac{1}{\theta^{\frac{s-q}{q-1}}}.$ Then by Corollary \ref{ATRcorollary1} and Lemma \ref{ATRlemma2},  the $A[\underline{t}_s]$-module $\frac{H_s}{M}$ is a finitely generated and torsion $\mathbb F_q[\underline{t}_s]$-module. Let $F(t_1, \ldots , t_s)\in A[\underline{t}_s]\setminus \{ 0\}$ such that $F(t_1, \ldots, t_s)\frac{H_s}{M}=\{ 0\}.$ Now, let $\chi$ be a Dirichlet character of type $s$ and conductor $f$ such that $F(\zeta_1(\chi), \ldots, \zeta_s(\chi))\not =0.$ Let $ev_{\chi} : \mathbb T_s(K_{\infty}) \rightarrow K_{\infty}(\zeta_1(\chi), \ldots, \zeta_s(\chi)), f(t_1, \ldots, t_s)\mapsto f(\zeta_1(\chi), \ldots, \zeta_s(\chi)).$ By \cite{APTR}, Lemma 9.1 and Proposition 9.2,  $ev_{\chi}$ induces a surjective morphism  of $A$-modules:
$$\psi : H_s\rightarrow H_{\chi} .$$
This morphism is induced by  the map:
$$\begin{array}{c}
\mathbb T_s(K_{\infty}) \rightarrow g(\chi) K_{\infty}(\zeta_1(\chi), \ldots, \zeta_s(\chi))\\
f\mapsto g(\chi)ev_{\chi}(f),\end{array}$$
where by \cite{APTR}, section 9,  $e_{\chi}((K(\lambda_f)\otimes_{\mathbb F_q}K_{\infty})\otimes_{\mathbb F_q}\mathbb F_q(\zeta_1(\chi), \ldots, \zeta_s(\chi)))$ is identified with $g(\chi) K_{\infty}(\zeta_1(\chi), \ldots, \zeta_s(\chi)).$ But, by the choice of $\chi,$ $\psi$ induces a surjective morphism of $A$-modules :
$$M\rightarrow H_{\chi}.$$
The Theorem follows.
\end{proof}

\subsection{Pseudo-nullity}${}$\par
Let $s\geq 2$ and set :
$$\mathcal M_s=\{ \phi_a(u_s), a\in A[\underline{t}_s]\},$$
$$\sqrt{\mathcal M_s}=\{ b\in A[\underline{t}_s], \exists a\in A[\underline{t}_s]\setminus\{ 0\}, \phi_a(b)\in \mathcal M_s\}.$$
\begin{proposition}\label{ATRproposition3} ${}$\par
\noindent Let $s\geq 2, s\not \equiv 1\pmod{q-1}.$ Then $\sqrt{\mathcal M_s}=\exp_{\phi}(U_s).$ Furtermore $H_s$ is a torsion $\mathbb F_q[\underline{t}_s]$-module if and only if $\sqrt{\mathcal M_s}=\mathcal M_s.$
\end{proposition}
\begin{proof} Recall that, since $s\not \equiv 1\pmod{q-1},$ $\exp_{\phi}: \mathbb T_s(K_{\infty})\rightarrow \mathbb T_s(K_{\infty})$ is  injective (see \cite{APTR}, Remark 6.3). Since $\mathcal M_s\subset \exp_{\phi}(\mathbb T_s(K_{\infty})),$ we get $\sqrt{\mathcal M_s}\subset \exp_{\phi}(\mathbb T_s(K_{\infty})),$ and therefore $\sqrt{\mathcal M_s}=\exp_{\phi}(U_s).$ Thus, $\exp_{\phi}$ induces an isomorphism of $A$-modules:
$$\frac{U_s}{L(t_1, \ldots, t_s)A[\underline{t}_s]}\simeq \frac{\sqrt{\mathcal M_s}}{\mathcal M_s}.$$
The Proposition is then a consequence of \cite{APTR}, Remark 9.13.
\end{proof}
\noindent A natural question is wether or not $H_s$ is a torsion $\mathbb F_q[\underline{t}_s]$-module when $s\not \equiv 1\pmod{q-1}$ (see \cite{APTR}, Question 9.10). First, we observe the following :

\begin{proposition}\label{ATRproposition4} 
\noindent Let $s\geq2, $ $s\equiv 1\pmod{q-1}.$ We have:
$$\sqrt{\mathcal M_s}=\mathcal M_s.$$
\end{proposition} 
\begin{proof}
We keep the notation used in the proof of Theorem \ref{ATRtheorem1}. For simplicity set $\mathcal N=\exp_{\phi}(\mathcal K)=\{ x\in \mathcal K, v_{\infty}(x)>\frac{s-q}{q-1}\}=\exp_{\phi}(\mathcal N),$ and set $E=\frac{\mathcal K}{\mathcal N}$ viewed as an $R$-module via $\phi.$ Set $\mathcal U_s=\mathbb F_q(\underline{t}_s)U_s=\{ x\in \mathcal K, \exp_{\phi}(x)\in R\}.$ Then $\exp_{\phi}^{(1)}$ induces an injective morphism of $A$-modules : $\mathcal U_s\rightarrow E.$ Recall that we have the following equality in $E:$
$$\exp_{\phi}^{(1)}(L(t_1, \ldots,t_s))= u_s.$$
Also note that:
$$\mathcal K= \mathcal U_s\oplus \mathcal N.$$
Let $x\in \mathcal K$ and let $c\in R\setminus\mathbb F_q(\underline{t}_s)$ such that $\phi_c(x)=0$ in $E.$ Then  there exists $y\in \mathcal K$ such that 
$\phi_c(x)=\exp_{\phi}(y).$ But this implies that (see \cite{APTR}, proof of Corollary 6.5) :
$$x\in \exp_{\phi}(\frac{y}{c}) + \exp_\phi({\frac{1}{c}\mathcal U_{s}}).$$
We get $x=0$ in $E.$ Therefore, for $c\in R\setminus \{0\},$ $\phi_c:E\rightarrow E$ is an injective morphism of $A$-modules.\par

\noindent
Now let $b\in \sqrt{\mathcal M_s}\setminus\{ 0\}, $ then there exists $a,c \in A[\underline{t}_s]\setminus \{0\},$ such that:
$$\phi_a(b)=\phi_c(u_s).$$
By the above discussion, we get in $E$ :
$$b=\exp_{\phi}^{(1)}(\frac{cL(t_1, \ldots, t_s)}{a}).$$
Write $c=da+r$ with $d,r\in A[\underline{t}_s]$ and $\deg_\theta(r)<\deg_\theta(a)$, and observe that :
$$\frac{cL(t_1, \ldots, t_s)}{a} = dL(t_1, \ldots, t_s) + \frac{rL(t_1, \ldots, t_s)}{a} \in \mathcal K= \mathcal U_s\oplus \mathcal N.$$
Thus, since $\exp_{\phi}^{(1)}(\mathcal N)\subset \mathcal N,$ by the proof of Theorem \ref{ATRtheorem1}:
$$\frac{rL(t_1, \ldots, t_s)}{a}\in L(t_1, \ldots,t_s)R\oplus \mathcal N.$$
Since $\deg_\theta(r)<\deg_\theta(a)$, $\frac{rL(t_1, \ldots, t_s)}{a} \in \mathcal N,$ thus, using $\exp_{\phi}^{(1)}(\mathcal N) \subset \mathcal N$, we get in $E$ : 
$$b=\exp_{\phi}^{(1)}(dL(t_1, \ldots, t_s)) =\phi_d(u_s)$$
Since $b, u_s \in A[\underline{t}_s],$ this implies the equality in $\mathcal K:$ $b=\phi_d(u_s),$ i.e. $b\in \mathcal M_s.$
\end{proof}

\begin{theorem}\label{ATRtheorem3}
Let $s\geq 2, $ then:
$$\sqrt{\mathcal M_s}=\mathcal M_s.$$
\end{theorem}
\begin{proof} We will use a similar argument as that used in the proof of Lemma \ref{ATRlemma3}. We can assume that $s\not \equiv 1\pmod{q-1}$ and let $r\in \{ 2, \ldots, q-1\}$ 
such that $s\equiv r\pmod{q-1}.$ We assume $s>r,$ the case $s=r$ being obvious. Set :
$$N=\{ x\in \mathbb T_s(K_{\infty}), v_{\infty}(x)\geq \frac{s-r}{q-1}\}.$$
Then, since $\frac{s-r}{q-1}>\frac{s-q}{q-1},$  we have :
$$N=\exp_{\phi}(N)\subset \exp_{\phi}(\mathbb T_s(K_{\infty})).$$
Observe that:
$$\mathbb T_s(K_{\infty})= \frac{1}{\theta^{\frac{s-r}{q-1}-1}}A[\underline{t}_s]\oplus N.$$
Set :
$$E=\frac{\mathbb T_s(K_{\infty})}{A[\underline{t}_s]+N}.$$
Observe that $E$ is a free and finitely generated $\mathbb F_q[\underline{t}_s]$-module of rank $\frac{s-r}{q-1}-1.$ By  Proposition \ref{ATRproposition3}, in order to prove the Theorem, it is enough to prove that there exists $\lambda \in \mathbb T_s(K_{\infty}),$ $v_{\infty}(\lambda)=0,$ such that if we write in $E:$
$$k\in \{ 1, \ldots, \frac{s-r}{q-1}-1\}, \, \exp_{\phi}(\theta^{-k}\lambda)=\sum_{j=1}^{\frac{s-r}{q-1}-1} \beta_{j,k}(\lambda)\theta^{-j}, \, \beta_{j,k}(\lambda)\in \mathbb F_q[\underline{t}_s],$$
then $det((\beta_{j,k}(\lambda))_{j,k})\in \mathbb  F_q[\underline{t}_s]\setminus \{0\}.$\par

\noindent Let $\psi:A\rightarrow \mathbb T_s(K_{\infty})\{Ê\tau \}$ be the morphism of $\mathbb F_q$-algebras given by :
$$\psi_{\theta}= (t_1-\theta)\cdots (t_r-\theta)\tau +\theta.$$
Let $\exp_{\psi}= 1+\sum_{j\geq 1} \frac{b_j(t_1)\cdots b_j(t_r)}{D_j}\tau^j .$ Recall that (Lemma \ref{lemma5}):
$$\exp_{\psi}(L(t_1, \ldots, t_r))=1.$$
Set:
$$\eta=\frac{1}{\omega(t_{r+1})\cdots \omega(t_s)}\in \mathbb T_s(K_{\infty})^\times.$$
Observe that $v_{\infty}(\eta)=\frac{s-r}{q-1}.$ 
We have :
$$\forall x \in \mathbb T_s(K_{\infty}), \exp_{\phi}(\eta x)=\eta\exp_{\psi}(x).$$ 
Thus, for $k\geq 0:$
$$\exp_{\phi}(\theta^k  L(t_1, \ldots, t_r)\eta)=\eta {\psi_{\theta^k}(1)}.$$
Now, observe that :
$$\exp_{\phi}(\theta^k  L(t_1, \ldots, t_r)\eta)\mid_{t_1=0, \ldots, t_s=0}\in (-\theta)^{\frac{s-r}{q-1}}\theta^k+(-\theta)^{k+1-\frac{s-r}{q-1}}A.$$
Now, if we set $\lambda = \theta ^{-\frac{s-r}{q-1}} L( t_1,\ldots, t_r) \eta,$ $\lambda\in \mathbb T_s(K_{\infty}),$ $v_{\infty}(\lambda)=0,$ and (see the proof of Lemma \ref{ATRlemma3}):
$$det((\beta_{j,k}(\lambda))_{j,k})\mid_{t_1=0, \ldots, t_s=0}\not =0.$$
\end{proof}
\noindent We deduce from the above Theorem by a similar argument as that of the proof of Theorem \ref{ATRtheorem2}:
\begin{corollary}\label{ATRcorollary2}${}$\par
\noindent Let $s$ be an integer, $s\geq 2,$ $s\not \equiv 1\pmod{q-1}.$ Then for almost all Dirichlet characters $\chi$ of type $s,$ $H_{\chi}=\{ 0\}.$ 
\end{corollary}

${}$\par

  \end{document}